\documentclass{amsart}
\usepackage{amsmath}
\usepackage{amssymb}
\usepackage{graphicx}
\usepackage{subfigure}
\usepackage[pdftex,bookmarksnumbered,bookmarksopen,colorlinks,linkcolor=red,anchorcolor=black,citecolor=blue,urlcolor=blue]{hyperref}
\usepackage{color}
\usepackage[ruled,vlined, boxed]{algorithm2e}

\newtheorem{theorem}{Theorem}[section]

\newtheorem{lemma}{Lemma}[section]
\newtheorem{remark}{Remark}[section]

   \long\def\comment#1{}

\def\ad#1{\begin{aligned}#1\end{aligned}}  
 \def\an#1{\begin{align}#1\end{align}}
 \def\d{\operatorname{div}}
\def\p#1{\begin{pmatrix}#1\end{pmatrix}} 
  \numberwithin{equation}{section}
\numberwithin{table}{section}
\numberwithin{figure}{section}

\def\boxit#1{\vbox{\hrule height1pt \hbox{\vrule width1pt\kern1pt
     #1\kern1pt\vrule width1pt}\hrule height1pt }}

\newcommand{\disp}{\displaystyle}

\newcommand{\bq}{\begin{equation}}
\newcommand{\eq}{\end{equation}}

\newcommand{\curl}{\operatorname{curl}}
\renewcommand{\div}{\operatorname{div}}

\usepackage{comment,enumerate,multicol,xspace}

  \newcounter{mnote}
  \let\oldmarginpar\marginpar
    \renewcommand\marginpar[1]{\-\oldmarginpar[\raggedleft\footnotesize #1]%
    {\raggedright\footnotesize #1}}

\begin{document}

\title[A posteriori error estimate for linear elasticity problems]
{ Residual-Based A Posteriori Error Estimates for Symmetric Conforming Mixed Finite Elements for Linear Elasticity Problems }
\author{Long Chen}%
\address{Department of Mathematics, University of California at Irvine, Irvine, CA 92697, USA}%
\email{chenlong@math.uci.edu}%
\author{Jun Hu}%
\address{LMAM and School of Mathematical Sciences, Peking University, Beijing 100871, China}%
\email{hujun@math.pku.edu.cn}%
\author{Xuehai Huang}%
\address{College of Mathematics and Information Science, Wenzhou University, Wenzhou 325035, China}%
\email{xuehaihuang@gmail.com}%
\author {Hongying Man}%
\address{School of Mathematics and Statistics, Beijing Institute of Technology, Beijing 100081, China}%
\email{manhy@bit.edu.cn}%
\thanks{The first author was supported by  NSF Grant DMS-1418934. This work was finished when L. Chen visited Peking University in the fall of 2015. He would like to thank Peking University for the support and hospitality, as well as for their exciting research atmosphere.}
\thanks{The second author was supported by  the NSFC Projects 11625101, 91430213 and 11421101.}
\thanks{The third author was supported by the NSFC Projects 11301396 and 11671304, Zhejiang Provincial
Natural Science Foundation of China Projects LY17A010010, LY15A010015 and LY15A010016, and Wenzhou Science and Technology Plan Project G20160019.}
\thanks{The last author was supported by  the NSFC Project 11401026. She would like to thank the support of the China Scholarship Council and the university of California, Irvine during her visit to UC Irvine from 2014 to 2015.}

\begin{abstract}
 A posteriori error estimators for the symmetric mixed finite element methods for linear
elasticity problems of Dirichlet and mixed boundary conditions are proposed. Stability and efficiency of the estimators are proved. Finally, we provide numerical examples
to verify the theoretical results.
\end{abstract}
\maketitle

\section{Introduction}

In this paper, we are concerned with the development of residual-based {\it a posteriori} error estimators for the symmetric mixed finite element methods for planar linear elasticity problems. Let $\Omega\subset\mathbb{R}^2$ be a bounded polygonal domain with boundary $\Gamma:=\partial\Omega$, based on the Hellinger-Reissner principle, the linear elasticity
  problem with homogeneous Dirichlet boundary condition within a stress-displacement form reads:
Find $(\sigma,u)\in\Sigma\times V :=H({\rm div},\Omega;\mathbb
{S})
        \times L^2(\Omega;\mathbb{R}^2)$, such that
\begin{equation}\label{eqn1}
\left\{ \ad{
  (A\sigma,\tau)+({\rm div}\tau,u)&= 0 && \hbox{for all \ } \tau\in\Sigma,\\
   ({\rm div}\sigma,v)&= (f,v) &\qquad& \hbox{for all \ } v\in V, }
   \right.
\end{equation}
where $\mathbb S\subset\mathbb R^{2\times2}$ is the space of symmetric matrices, and the symmetric tensor space for stress and the
   space for vector displacement are, respectively,
  \an{
  H({\rm div},\Omega;\mathbb {S})
    &:= \Big\{ \p{\tau_{ij} }_{2 \times 2} \in H(\d, \Omega)
    \ \Big| \ \tau_{12} = \tau_{21} \Big\}, \\
     L^2(\Omega;\mathbb{R}^2) &:=
     \Big\{ \p{u_1, u_2}^T
          \ \Big| \ u_1,u_2 \in L^2(\Omega) \Big\}  .}
Throughout the paper, the compliance tensor
$A:\mathbb{S}~\rightarrow~\mathbb{S}$, characterizing the
properties of the material, is bounded and symmetric positive
definite. In the homogeneous isotropic case, the compliance tensor is given by $A\tau=(\tau-\lambda/(2\mu+2\lambda){\rm tr}\tau~ {\rm I})/(2\mu)$, where $\mu>0, \lambda\geq 0$ are the Lam\'{e} constants, ${\rm I}$ is the identity matrix, ${\rm tr}\tau=\tau_{11}+\tau_{22}$ is the trace of the matrix $\tau$. For simplicity, we assume $A$ is a constant matrix in this paper and comment on the generalization to the piecewise constant matrix case.


Because of the symmetry constraint on the stress tensor, it is extremely difficult to construct
stable conforming finite elements of (1.1) even for 2D problems, as stated in the plenary presentation
to the 2002 International Congress of Mathematicians by Arnold~\cite{Arnold2002}. An important progress
 in this direction is the work of Arnold and Winther~\cite{ArnoldWinther2002} and Arnold, Awanou, and Winther~\cite{ArnoldAwanouWinther2008}. In particular, a sufficient condition of the discrete stable method is proposed in these two papers,
which states that a discrete exact sequence guarantees the stability of the mixed method. Based
on such a condition, conforming mixed finite elements on the simplical and rectangular meshes are developed for both 2D and 3D~\cite{AdamsCockburn2005, ArnoldAwanou2005, ArnoldWinther2003a, ChenWang2011, HuShi2007a}. Recently, based on a crucial structure of symmetric matrix valued
piecewise polynomial $H(\div)$ space and two basic algebraic results, Hu and Zhang developed a
new framework to design and analyze the mixed finite element of elasticity problems. As a result, on both
simplicial and tensor product grids, several families of both symmetric and optimal mixed
elements with polynomial shape functions in any space dimension are constructed, see more details in~\cite{Hu2015a, Hu2015, HuZhang2014, HuZhang2015, HuZhang2016}. Theoretical and numerical analysis show that symmetric mixed finite element method is a  popular choice for a robust stress approximation~\cite{CarstensenEigelGedicke2011, CarstensenGuntherReininghausThiele2008}.

Computation with adaptive grid refinement has proved to be a useful and efficient tool
in scientific computing over the last several decades. When the domain contains a re-entering corner, the stress has a singularity at that corner,
non-uniform mesh is necessary to catch the singularity. Adaptive finite element methods based on local mesh refinement can recovery the optimal rate
of convergence. The key behind this technique is to
design a good {\it a posteriori} error estimator that provides a guidance on how and where grids
should be refined. The residual-based {\it a posteriori} error estimators provide indicators for refining and coarsening the mesh and allow to control whether the error is below a given
threshold. Various error estimators for mixed finite element discretizations of the Poisson equation have been obtained
in~\cite{Alonso1996,Carstensen1997,ChenHolstXu2009,GaticaMaischak2005,HoppeWohlmuth1997,LarsonMaalqvist2008,LovadinaStenberg2006}. Extension to the mixed finite element for
linear elasticity is, however, very limited. In~\cite{CarstensenDolzmann1998, Kim2012c, LonsingVerfurth2004}, the authors gave the {\it a posteriori} error estimators for the nonsymmetric mixed finite elements only.

The symmetry of the stress tensor brings essential difficulty to the {\it a posteriori} error analysis. Since only the symmetric part is approximated and not the full gradient, the approach of {\it a posteriori} error analysis developed in~\cite{CarstensenDolzmann1998,CarstensenHu2007,Kim2012c, LonsingVerfurth2004} cannot be applied directly. In order to overcome this difficulty, Carstensen and Gedicke propose to generalize the framework of the {\it a posteriori} analysis for nonsymmetric mixed finite elements to the case of symmetric elements by decomposing the stress into the gradient and the asymmetric part of the gradient. A robust residual-based {\it a posteriori} error estimator for Arnold-Winther's symmetric element was proposed in \cite{CarstensenGedicke2016}, but an arbitrary asymmetric approximation $\gamma_h$ of the asymmetric part of the gradient skew$(Du)=(Du-D^Tu)/2$ was involved in this estimator. Furthermore $\gamma_h$ was chosen as the asymmetric gradient of a post-processed displacement to ensure the efficiency of the estimator. More details can be found below.

The goal of this paper is to present an {\it a posteriori} error estimator together with a theoretical upper and lower bounds, for the conforming
and symmetric mixed finite element solutions developed in~\cite{ArnoldWinther2002,HuZhang2014}. We shall follow the guide  principle in \cite{ArnoldWinther2002}: use the continuous and discrete linear elasticity complex, c.f. \eqref{deRham} and \eqref{deRhamdiscrete}.

Given an approximation $\sigma_h$ on the triangulation $\mathcal T_h$ consisting of triangles,
we construct the following {\it a posteriori} error estimator, denoted by $\eta$,
 $$
 \eta^2(\sigma_{h}, {\mathcal{T}_h}):=\sum\limits_{K\in\mathcal{T}_h}\eta_K^2(\sigma_h)+\sum\limits_{e\in\mathcal{E}_h}\eta_e^2(\sigma_h)
  $$
where
$$
\eta_K^2(\sigma_h) :=h_K^4\|{\rm curl \, curl \,}(A\sigma_h)\|_{0,K}^2,~~~
\eta_e^2(\sigma_h):=h_e\|\mathcal{J}_{e,1}\|_{0,e}^2+h_e^3\|\mathcal{J}_{e,2}\|_{0,e}^2,
$$
$$
\begin{array}{ll}
 \disp \mathcal{J}_{e,1}:=&\left\{\begin{array}{lr}
 \disp\big[(A\sigma_h)t_e\cdot t_e\big]_e&\quad\quad\quad\quad\quad\quad\quad\ \  {\rm if~} e\in\mathcal{E}_h(\Omega),\\
 \disp\big((A\sigma_h)t_e\cdot t_e\big)|_e&~~~~ {\rm if~} e\in\mathcal{E}_h(\Gamma),
 \end{array}\right. \\
 \\
 \disp \mathcal{J}_{e,2}:=&\left\{\begin{array}{lr}
\disp \big[{\rm curl}(A\sigma_h)\cdot t_e\big]_e&~~~~ {\rm if~} e\in\mathcal{E}_h(\Omega),\\
\disp \big({\rm curl}(A\sigma_h)\cdot t_e-\partial _{t_e}\big((A\sigma_h)t_e\cdot \nu_e\big)\big)|_e&~~~~ {\rm if~} e\in\mathcal{E}_h(\Gamma),
 \end{array}\right.
 \end{array}
$$
with $\mathcal E_h$ being the collection of all edges of $\mathcal T_h$. We write $\mathcal E_h=\mathcal E_{h}(\Omega)\bigcup\mathcal {E}_h(\Gamma)$, where $\mathcal E_{h}(\Omega)$ is the collection of interior edges and $\mathcal E_{h}(\Gamma)$ is the collection of  all element edges on the boundary. For any edge $e\in\mathcal{E}_h$, let $t_e=(-n_2,n_1)^T$ be the unit tangential vector along edge $e$ for the unit outward normal $\nu_e=(n_1,n_2)^T$. Let $h_K$ be the diameter of the element $K$ and $h_e$ be the length of edge $e$.  The data oscillation is defined as
$${\rm osc}^2(f,\mathcal{T}_h):=\sum\limits_{K\in\mathcal{T}_h}h_K^2\|f-Q_h f\|_{0,K}^2,$$
where $Q_h$ is the $L^2$ orthogonal projection operator onto the discrete displacement space.

Using the Helmholtz decomposition induced from the linear elasticity complex \cite{ArnoldWinther2002, CarstensenDolzmann1998}, we establish the following reliability
$$
\|\sigma-\sigma_h\|_A\leq C_1(\eta(\sigma_{h}, {\mathcal{T}_h})+{\rm osc}(f,\mathcal{T}_h)).
$$
In addition, we will prove the following efficiency estimate
$$C_2\eta(\sigma_{h}, {\mathcal{T}_h})\leq \|\sigma-\sigma_h\|_A$$
by following the approach from ~\cite{Alonso1996}.

We also generalize the above results to the mixed boundary problems, for which the error estimator is modified on the Dirichlet boundary edges. Reliability and efficiency of the modified error estimator can be proved similarly.

In \cite{ChenHuHuang2016}, a superconvergent approximate displacement $u_h^*$ was constructed by a postprocessing of $(\sigma_h,u_h)$ . Using this result and the {\it a posteriori} error estimation of the stress, we also give the {\it a posteriori} error estimation for the displacement $\|u-u_h^*\|_{1,h}$ in a mesh dependent norm.
\medskip

In order to compare with the {\it a posteriori} error estimator in \cite{CarstensenGedicke2016}, we present their estimator as follows:
\begin{align*}
\tilde{\eta}^2(\sigma_h,\mathcal{T}_h)&:={\rm osc}^2(f,\mathcal{T}_h)+{\rm osc}^2(g,\mathcal{E}_h(\Gamma_N))\\
&+\sum\limits_{K\in\mathcal{T}}h_K^2\|{\rm curl}(A\sigma_h+\gamma_h)\|_{0,K}^2\\
&+\sum\limits_{e\in\mathcal{E}_h(\Omega)}h_e\|[A\sigma_h+\gamma_h]_e\tau_e\|_{0,e}^2\\
&+\sum\limits_{e\in\mathcal{E}_h(\Gamma_D)}h_e\|(A\sigma_h+\gamma_h-\nabla u_D)\tau_e\|_{0,e}^2.
\end{align*}
(The estimator is rewritten in our notation and the details of the standard notation can be found below.) To ensure the efficiency of the estimator, a sufficiently accurate polynomial asymmetric approximation $\gamma_h$ of the asymmetric gradient ${\rm skew}(Du):=(Du-D^Tu)/2$ is involved in the above estimator.
Since the global approximation or even minimization may be too costly, Carstensen and Gedicke compute the sufficiently accurate approximation $\gamma_h={\rm skew}(Du_h^*)$ by the post-processed displacement $u_h^*$ in the spirit of Stenberg \cite{Stenberg1988}. As we can see, this estimator is totally different to ours. The estimators we propose use the symmetric stress directly and do not need any estimation of the asymmetric part. Therefore it is more computationally efficient.

The remaining parts of the paper is organized as follows. Section 2 presents the notations and the discrete finite element problems. Section 3 proposes an {\it a posteriori} error estimator for the stress and proves the reliability and efficiency of the estimator. Section 4 generalizes the results of section 3 to mixed boundary problems.  Section 5 gives {\it a posteriori} error estimation for the displacement. Section 6 presents numerical experiments to show the effectiveness of the estimator.
Throughout this paper, we use ``$\lesssim\cdots $" to mean that ``$\leq C\cdots$", where $C$ is a generic positive constant independent of $h$ and the Lam\'{e} constant $\lambda$, which may take different values at different appearances.

\section{Notations and preliminaries}

Standard notations on Sobolev spaces and norms are adopted throughout this paper and, for brevity, $\|\cdot\|:=\|\cdot\|_{L^2(\Omega)}$ denotes the $L^2$ norm. $(\cdot,\cdot)_K$ represents, as usual, the $L^2$ inner product on the domain $K$, the subscript $K$ is omitted when $K=\Omega$. $\langle\cdot,\cdot\rangle_{\Gamma}$ represents the $L^2$ inner product on the boundary $\Gamma$. For brevity, let $\disp\partial_{x_i}:=\partial/\partial x_i$ and $\partial^2_{x_ix_j}:=\partial^2/\partial x_i\partial x_j, j=1,2,$  $\partial_{\nu}:=\partial/\partial \nu$, $\partial_{t}:=\partial/\partial t$. For $\phi\in H^1(\Omega;\mathbb{R})$,  $v=(v_1,v_2)^{T}\in H^1(\Omega;\mathbb{R}^2)$, set
$$
{\rm \bf{Curl}}\phi:=\left(-\partial \phi/ \partial x_2,\ \partial \phi/ \partial x_1\right),\quad\quad\quad
{\rm \bf{Curl}}v:=\left(\begin{array}{cc}
-\partial v_1/ \partial x_2&\partial v_1/ \partial x_1\\
-\partial v_2/ \partial x_2&\partial v_2/ \partial x_1
\end{array}\right).$$
For $\tau=(\tau_{i,j})_{2\times2}\in H^1(\Omega; \mathbb{R}^{2\times 2})$, set
$$
{\rm curl}\tau:=\left(\begin{array}{c}
\partial\tau_{12}/ \partial x_1-\partial\tau_{11}/ \partial x_2\\
\partial\tau_{22}/ \partial x_1-\partial\tau_{21}/ \partial x_2
\end{array}\right),\qquad
{\rm div}\tau:=\left(\begin{array}{c}
\partial\tau_{11}/ \partial x_1+\partial\tau_{12}/ \partial x_2\\
\partial\tau_{21}/ \partial x_1+\partial\tau_{22}/ \partial x_2
\end{array}\right).
$$
Namely the differential operators $\curl$ and $\div$ are applied rowwise for tensors.

Let $\mathcal{T}_h$ be a shape-regular triangulation of $\bar{\Omega}$ into triangles with the set of edges $\mathcal{E}_h$. Denote by $\mathcal{E}_h(\Omega)$ the collection of all interior element edges in $\mathcal T_h$ and $\mathcal E_h(\Gamma)$ the collection of all element edges on the boundary. For any triangle $K\in\mathcal{T}_h$, let $\mathcal{E}(K)$ be the set of its edges. For any edge $e\in\mathcal{E}(K)$, let $t_e=(-n_2,n_1)^T$ be the unit tangential vector along edge $e$ for the unit outward normal vector $\nu_e=(n_1,n_2)^T$, $h_K$ be the diameter of the element $K$ and  $h_e$ be the length of the edge $e$, $h=\max\limits_{K\in\mathcal{T}_h}\{h_K\}$ be the diameter of the partition $\mathcal{T}_h$. The jump $[w]_e$ of $w$ across edge $e=\bar{K}_+\cap\bar{K}_-$ reads
$$[w]_e:=(w|_{K_+})_e-(w|_{K_-})_e.$$
Particularly, if $e\in \mathcal{E}_h(\Gamma),\ [w]_e:=w|_e$.

Let $\Sigma_h\times V_h\subseteq \Sigma \times V $ be a symmetric conforming mixed element defined on the mesh $\mathcal{T}_h$, then the discrete mixed formulation for the problem (1.1) is: find $(\sigma_h, u_h)\in \Sigma_h\times V_h$, such that
\begin{equation}\label{mfem}
\left\{
\begin{array}{rcl}(A\sigma_h,\tau_h)+({\rm div}\tau_h, u_h)&& =0\quad\quad\quad\quad\;{\rm for ~all~}\tau_h\in \Sigma_h,\\
({\rm div}\sigma_h,v_h)&& =(f,v_h)\quad\quad{\rm for~all~}v_h\in V_h.
\end{array}
\right.
\end{equation}

In the sequel, we briefly introduce Hu-Zhang element~
\cite{Hu2015a, HuZhang2014, HuZhang2016}.
For each $K\in\mathcal{T}_h$, let $P_k(K)$ be the space of polynomials of total degree at most $k$ on $K$ and
$$P_{k}(K; \mathbb{S}):=\{\tau\in L^2(K;\mathbb{R}^{2\times 2})|\tau_{i,j}\in P_k(K), \tau_{ij}=\tau_{ji}, 1\leq i\leq 2,1\leq j\leq 2\}, $$
$$P_{k}(K; \mathbb{R}^2):=\{v\in L^2(K;\mathbb{R}^2)|v_i\in P_k(K), 1\leq i\leq 2\}, $$
define an $H(\div, K; \mathbb{S})$ bubble function  as
\[
B_{K,k}:=\left\{\tau\in P_{k}(K; \mathbb{S}): \tau\nu|_{\partial K}=0\right\}.
\]
The Hu-Zhang element space
is given by
\begin{align*}
\Sigma_{h}&:=\widetilde{\Sigma}_{k,h} + B_{k,h},\\
V_{h}&:=\left\{v\in L^2(\Omega; \mathbb{R}^2): v|_K\in P_{k-1}(K; \mathbb{R}^2)\quad \forall\,K\in\mathcal
{T}_h\right\},
\end{align*}
with integer $ k\geq 3$, where
\begin{align*}
B_{k,h}&:=\left\{\tau\in H(\div, \Omega; \mathbb{S}):
\tau|_K\in B_{K,k} \quad \forall\,K\in\mathcal{T}_h \right\}, \\
\widetilde{\Sigma}_{k,h}&:=\left\{\tau\in H^1(\Omega; \mathbb{S}):
\tau|_K\in P_{k}(K; \mathbb{S}) \quad \forall\,K\in\mathcal{T}_h \right\}.
\end{align*}


For the above elements, the following {\it a priori } error estimate holds.
\begin{theorem}[A priori error estimate \cite{Hu2015a, HuZhang2014, HuZhang2016}]
The exact solution $(\sigma,u)$ of problem (1.1) and the approximate solution $(\sigma_h, u_h)$ of problem (2.1) satisfy
\begin{align*}
\|\sigma-\sigma_h\|_0&\lesssim h^m\|\sigma\|_m, \quad\quad\;\,\text{for}~1\leq m\leq k+1,\\
\|{\rm div}(\sigma-\sigma_h)\|_0&\lesssim h^m\|{\rm div}\sigma\|_m,\quad\text{for}~0\leq m\leq k,\\
\|u-u_h\|_0&\lesssim h^m\|u\|_{m+1}, \quad\;\,\text{for}~1\leq m\leq k.
\end{align*}
\end{theorem}

In the continuous case, the following exact sequence
\begin{equation}\label{deRham}
P_1(\Omega)
\longrightarrow
H^2(\Omega)
\stackrel{\rm \bf{Curl} \,  \bf{Curl}}{\longrightarrow}
H(\div,\Omega;\mathbb S)
\stackrel{\div}{\longrightarrow}
L^2(\Omega,\mathbb R^2)
\end{equation}
holds for linear elasticity~\cite{ArnoldWinther2002}. In the discrete case, the exact sequence holds similarly
\begin{equation}\label{deRhamdiscrete}
P_1(\Omega)
\longrightarrow
\Phi_h
\stackrel{\rm \bf{Curl} \,  \bf{Curl}}{\longrightarrow}
\Sigma_h
\stackrel{\div}{\longrightarrow}
V_h.
\end{equation}
As stated in \cite{ArnoldWinther2002}, the space $\Phi_h$ for the Arnold-Winther element is precisely the space of $C^1$
piecewise polynomials which are $C^2$ at the vertices, that is, the well-known high-order Hermite
 or Argyris finite element. The Hu-Zhang element is an enrichment of the Arnold-Winther element, adding all the piecewise polynomial  matrices of degree $k$ which are not divergence-free on each element and belong to $ H(\div,\Omega;\mathbb S)$ globally. So the space $\Phi_h$ for the Hu-Zhang element is the same as the one for the Arnold-Winther element.

\begin{lemma}[Helmholtz-type decomposition \cite{ArnoldWinther2002, CarstensenDolzmann1998}]
For any $\tau \in L^2(\Omega; \mathbb S)$, there exists $v\in H_0^1(\Omega;\mathbb{R}^2)$ and $\phi \in H^2(\Omega)/P_1(\Omega)$, such that
\begin{equation}\label{symdec}
\tau = \mathcal C\varepsilon (v) + {\rm \bf{Curl}\, \bf{Curl}} \phi,
\end{equation}
and the decomposition is orthogonal in the weighted $L^2$-inner product $(\mathcal C^{-1} \cdot,\cdot): = (A~ \cdot,\cdot)$, i.e.,
\begin{equation}
\|\tau\|_A^2 = \|\varepsilon (v)\|_{A^{-1}}^2 + \| {\rm \bf{Curl} \, \bf{Curl}} \phi \|_A^2,
\end{equation}
where $P_1(\Omega)$ is the linear polynomial space on $\Omega$, the norm $\|\cdot\|_A=(A~\cdot,\cdot)$. \end{lemma}

Since $$
(A^{-1}A\tau,\tau)=(\tau,\tau)=(A(A^{-1}\tau),\tau),
$$
by the boundedness and coerciveness of the operator $A$, we obtain the following relationship of the norms: for any $\tau\in\Sigma$, there exist positive constants $C_1$ and $C_2$, which are independent of the Lam\'e constant $\lambda$, such that
\begin{align}\label{normequivalence}
C_2\|\tau\|_A^2= C_2(A\tau,\tau)\leq\|\tau\|_0^2\leq C_1(A^{-1}\tau,\tau)=C_1\|\tau\|_{A^{-1}}^2.
\end{align}

It is the goal of this paper to present {\it a posterior} error estimate of $\sigma - \sigma_h$ for the Hu-Zhang element method.
It is worth mentioning that the {\it a posterior} error estimator designed in this paper can be easily extended to the Arnold-Winther element~\cite{ArnoldWinther2002}.

\section{A posteriori Error Estimation for Stress}

In this section, we shall prove the reliability and efficiency of the error estimator. The main observation is that: although it is a saddle point problem, the error of stress $\sigma - \sigma _h$ is orthogonal to the divergence-free subspace, while the part of the error that is not divergence- free can be bounded by the data oscillation using the stability of the discretization.

For any $\tau_h\in\Sigma_h$, the error estimator is defined as
 \begin{equation}\label{estimator-1}
 \eta^2(\tau_{h}, {\mathcal{T}_h}):=\sum\limits_{K\in\mathcal{T}_h}\eta_K^2(\tau_h)+\sum\limits_{e\in\mathcal{E}_h}\eta_e^2(\tau_h),
 \end{equation}
where
$$
\eta_K^2(\tau_h) :=h_K^4\|{\rm curl \, curl \,}(A\tau_h)\|_{0,K}^2,~~~
\eta_e^2(\tau_h):=h_e\|\mathcal{J}_{e,1}\|_{0,e}^2+h_e^3\|\mathcal{J}_{e,2}\|_{0,e}^2,
$$
$$
\begin{array}{ll}
 \disp \mathcal{J}_{e,1}:=&\left\{\begin{array}{lr}
 \disp\big[(A\tau_h)t_e \cdot t_e\big]_e&\quad\quad\quad\quad\quad\quad\quad\ \  {\rm if~} e\in\mathcal{E}_h(\Omega),\\
 \disp\big((A\tau_h)t_e\cdot t_e\big)|_e&~~~~ {\rm if~} e\in\mathcal{E}_h(\Gamma),
 \end{array}\right. \\
 \\
 \disp \mathcal{J}_{e,2}:=&\left\{\begin{array}{lr}
\disp \big[{\rm curl}(A\tau_h)\cdot t_e\big]_e&~~~~ {\rm if~} e\in\mathcal{E}_h(\Omega),\\
\disp \big({\rm curl}(A\tau_h)\cdot t_e-\partial _{t_e}\big((A\sigma_h)t_e\cdot \nu_e\big)\big)|_e&~~~~ {\rm if~} e\in\mathcal{E}_h(\Gamma).
 \end{array}\right.
 \end{array}
$$
The data oscillation is defined as
$${\rm osc}^2(f,\mathcal{T}_h):=\sum\limits_{K\in\mathcal{T}_h}h_K^2\|f-Q_h f\|_{0,K}^2,$$
where $Q_h$ is the $L^2$ orthogonal projection operator onto the discrete displacement space $V_h$.
\subsection{Stability result}
For the easy of exposition, we write the mixed formulation for linear elasticity as $\mathcal L(\sigma, u) = f$. The natural stability of the operator is $\|\sigma\|_{H(\div)}+\|u\|\lesssim \|f\|.$ However, a stronger stability can be proved for a special perturbation of the data.

\begin{lemma}
Let $f_h$ be the $L^2$ projection of $f$ onto $V_h$ and let $(\sigma, u) = \mathcal L^{-1}f$ and $(\tilde \sigma, \tilde u)= \mathcal L^{-1}f_h$.
Then we have
\begin{equation}\label{oscf}
\|\sigma - \tilde \sigma\|_A \lesssim {\rm osc}(f,\mathcal T_h).
\end{equation}
\end{lemma}
\begin{proof}
Use the first equation of \eqref{eqn1} and let $v = u - \tilde u$
\begin{align*}
(A(\sigma - \tilde\sigma), \sigma - \tilde\sigma)&=-({\rm div}(\sigma-\tilde\sigma), u - \tilde u)=- (f- Q_h f, u - \tilde u)\\
&=(f-Q_h f, Q_h v - v)\\
&\leq \sum\limits_{K\in\mathcal{T}_h}\|f-Q_h f\|_{0,K}\|v-Q_hv\|_{0,K}\\
&\lesssim\sum\limits_{K\in\mathcal{T}_h}\|f-Q_h f\|_{0,K}h_K |v |_{1,K}\\
&\lesssim\left (\sum\limits_{K\in\mathcal{T}_h}h_K^2\|f-Q_h f\|_{0,K}^2\right )^{\frac{1}{2}}\|\varepsilon(v)\|_0,
\end{align*}
where the Korn's inequality is used and the symmetric gradient $\varepsilon(v)=\frac{1}{2}(\nabla v+(\nabla v)^T)$. Since $\varepsilon(v)=A(\sigma-\tilde{\sigma})$, by \eqref{normequivalence}, $\|\varepsilon(v)\|_0 \lesssim \|\sigma - \tilde \sigma\|_A$. We acquire the desirable stability result.
\end{proof}

The oscillation ${\rm osc}(f,\mathcal T_h)$ is an upper bound of $\|f-f_h\|_{-1}$ and is of high order comparing with the error estimator.

\subsection{Orthogonality}
For any $\phi\in H^2(\Omega)$, ${\rm \bf{Curl} \, \bf{Curl} \, }\phi\in H({\rm div},\Omega;\mathbb{S})$, we can use the exact sequence property $\div {\rm \bf{Curl} \, \bf{Curl} \, } = 0$ to get
\begin{equation}\label{Asigma}
(A\tilde \sigma, {\rm \bf{Curl} \, \bf{Curl} \, } \phi) =- (\tilde u, \div {\rm \bf{Curl} \, \bf{Curl} \,}\phi ) = 0.
\end{equation}
Similarly $$(A \sigma_h, {\rm\bf{Curl} \, \bf{Curl} \, } \phi_h) = -(u_h, \div {\rm \bf{Curl} \, \bf{Curl} \,}\phi_h ) = 0$$
for any $\phi_h\in \Phi_h$.
Therefore we have a partial orthogonality
\begin{equation}\label{partial-orthog}
(A(\tilde \sigma - \sigma _h), {\rm \bf{Curl} \, \bf{Curl} \,} \phi_h ) = 0 \quad \forall~\phi_h \in \Phi_h.
\end{equation}

\subsection{Upper bound}
Let $S_h^5$ denote the Argyris finite element space, which consists of $C^1$ piecewise polynomials of degree less than or equal to $5$
\begin{align*}
S_h^5:=\left\{v\in L^2(\bar{\Omega}):~v|_K\in P_5(K),~\forall K\in\mathcal{T}_h, ~v \text{ and its all first and second}\ \ \right.\\
 \text{   derivatives are continuous on the vertices,} ~v ~\text{is continuous}\ \ \\
\left.\text{  along the normal direction on the edge midpoints}\right\}.
 \end{align*}
Following \cite{ShiWang2013, GiraultScott2002}, we can define a quasi-interpolation operator $I_h:~H^2(\Omega)~\rightarrow S_h^5$, which preserves the values of the function on all vertices of $\mathcal{T}_h$. On each element $K\in\mathcal{T}_h$, for any $v\in H^2(\Omega)$, $I_hv|_K\in P_5(K)$ and it satisfies
\begin{itemize}
\item $I_hv|_K(a_{i,K})=v(a_{i,K}),\quad\quad 1\leq i\leq 3;$\\
\item $\partial_{x_j}(I_hv|_K)(a_{i,K})=\frac{1}{\mathcal{N}_h(a_{i,K})}\sum\limits_{K'\in S(a_{i,K})}\partial_{x_j}(P_hv|_{K'})(a_{i,K}),\ 1\leq i\leq 3,\\ \ j=1,2;$\\
\item $\partial^2_{x_jx_l}(I_hv|_K)(a_{i,K})=\frac{1}{\mathcal{N}_h(a_{i,K})}\sum\limits_{K'\in S(a_{i,K})}\partial^2_{x_jx_l}(P_hv|_{K'})(a_{i,K}),\ 1\leq i\leq 3,\ 1\leq j\leq l\leq 2;$\\
\item $\partial_{\nu}(I_hv|_K)(a_{3+i,K})=\frac{1}{\mathcal{N}_h(a_{3+i,K})}\sum\limits_{K'\in S(a_{3+i,K})}\partial_{\nu}(P_hv|_{K'})(a_{3+i,K}),\ 1\leq i\leq 3;$
\end{itemize}
where $a_{i,K},1\leq i\leq 3,$ are the vertices of $K$, $a_{3+i,K},1\leq i\leq 3,$ are the edge midpoints of $K$, $\nu$ is the edge outer normal of the element $K$ on the edge midpoint, $S(a_{i,K})=\bigcup\{K\in\mathcal{T}_h: a_{i,K}\in K\}$ and $\mathcal{N}_h(a_{i,K})=\text{card}\{K:K\in S(a_{i,K})\}$, $P_h$ is the projection operator from $L^2(\Omega)$ onto the piecewise linear polynomial finite element space on $\mathcal{T}_h$. It is obvious that the interpolation operator $I_h$ is uniquely determined by the above degrees of freedom. Furthermore, $I_h$ is a projection, i.e.
\begin{align}
I_hv=v\quad\quad \forall v\in S_h^5,
\end{align}
and it preserves the value of the function on vertices for any $v\in H^2(\Omega)$, i.e.
\begin{align}\label{preserve-vertex value}
I_hv(a_{i,K})=v(a_{i,K})\quad\quad \forall K\in\mathcal{T}_h,\ \ 1\leq i\leq 3.
\end{align}
A similar scaling argument as in \cite{ShiWang2013, GiraultScott2002} gives the following interpolation estimates
\begin{equation}\label{inter-error-K}
|v-I_hv|_{m,K} \lesssim h_K^{2-m}|v|_{2,S_K},~~0\leq m\leq 1,\ \ \forall K\in\mathcal{T}_h,
\end{equation}
\begin{equation}\label{inter-error-e}
|v-I_hv|_{m,e} \lesssim h_e^{2-m-\frac{1}{2}}|v|_{2,S_e},~~0\leq m\leq 1, \ \ \forall\ e\in\mathcal{E}_h,
\end{equation}
where $S_K=\bigcup\{K_i\in\mathcal{T}_h: K_i\bigcap\bar{K}\neq\varnothing\}$, $S_e=\bigcup\{K_i\in\mathcal{T}_h:K_i\bigcap e\neq\varnothing \}.$

Applying the Helmholtz decomposition to the error $\tilde \sigma-\sigma_h$, we have
\begin{equation}\label{HD-error}{
\tilde \sigma-\sigma_h=\mathcal{C}\varepsilon(v)+{\rm \bf{Curl} \, \bf{Curl} \,}\phi
}
\end{equation}
and
\begin{equation}\label{HD-error-estimate}
\|{\rm \bf{Curl} \, \bf{Curl} \,}\phi\|_A\leq\|\tilde{\sigma}-\sigma_h\|_A,
\end{equation}
where $v\in H_0^1(\Omega;\mathbb{R}^2)$ and $\phi\in H^2(\Omega)/P_1(\Omega)$.
By this orthogonal decomposition and the fact $\div (\tilde \sigma-\sigma_h) = 0$,
\begin{align*}
\| \tilde \sigma-\sigma_h\| _A^2&= (A( \tilde \sigma-\sigma_h),\mathcal{C}\varepsilon(v)+{\rm \bf{Curl} \, \bf{Curl} \,}\phi)\\
&= - (\div(\tilde \sigma-\sigma_h),v)+(A(\tilde \sigma-\sigma_h),{\rm \bf{Curl} \, \bf{Curl} \,}\phi)\\
&= (A(\tilde \sigma-\sigma_h),{\rm \bf{Curl} \, \bf{Curl} \,}\phi).
\end{align*}
Since ${\rm \bf{Curl} \, \bf{Curl} \,} (I_h\phi)\in \Sigma_h$, by the orthogonality \eqref{partial-orthog} and the equation \eqref{Asigma},
\begin{align*}
(A(\tilde \sigma-\sigma_h),{\rm \bf{Curl} \, \bf{Curl} \,}\phi) &= (A(\tilde \sigma-\sigma_h),{\rm \bf{Curl} \, \bf{Curl} \,}(\phi-I_h\phi))\\
&= -(A\sigma_h,{\rm \bf{Curl} \, \bf{Curl} \,}(\phi-I_h\phi)).
\end{align*}
An integration by parts gives
\begin{align}
& (A\sigma_h,{\rm \bf{Curl} \, \bf{Curl} \,}(\phi-I_h\phi))\nonumber \\
 &= -\sum\limits_{K\in\mathcal{T}_h}({\rm curl}(A\sigma_h),{\rm \bf{Curl}}(\phi-I_h\phi))_K+\sum\limits_{K\in\mathcal{T}_h}\langle(A\sigma_h)t,{\rm \bf{Curl}}(\phi-I_h\phi)\rangle_{\partial K}\nonumber\\
 &=\sum\limits_{K\in\mathcal{T}_h}({\rm curl \, curl \,}(A\sigma_h),\phi-I_h\phi)_K+\sum\limits_{K\in\mathcal{T}_h}\langle(A\sigma_h)t,{\rm \bf{Curl}}(\phi-I_h\phi)\rangle_{\partial K}\label{byparts}\\&-\sum\limits_{K\in\mathcal{T}_h}\langle{\rm curl}(A\sigma_h)\cdot t,\phi-I_h\phi\rangle_{\partial K}\nonumber.
\end{align}
The second term of the right hand side can be rewritten as
\begin{align*}
\sum\limits_{K\in\mathcal{T}_h}\langle A\sigma_h t,{\rm \bf{Curl}}(\phi-I_h\phi)\rangle_{\partial K}
 =&\sum\limits_{K\in\mathcal{T}_h}\langle(A\sigma_h)t\cdot t,{\rm \bf{Curl}}(\phi-I_h\phi)\cdot t\rangle_{\partial K}\\
 +&\sum\limits_{K\in\mathcal{T}_h}\langle(A\sigma_h t)\cdot \nu,{\rm \bf{Curl}}(\phi-I_h\phi)\cdot \nu\rangle_{\partial K}
\end{align*}
Since the compliance tensor
$A$ is symmetric and continuous,  $(A\sigma_h t)\cdot \nu=(A\sigma_h \nu)\cdot t$ and $(A\sigma_h t)\cdot \nu$ is continuous across the interior element edge. This implies
\begin{align*}
\sum\limits_{K\in\mathcal{T}_h}\langle(A\sigma_h t)\cdot \nu,{\rm \bf{Curl}}(\phi-I_h\phi)\cdot \nu\rangle_{\partial K} 
=&-\sum\limits_{e\in\mathcal{E}_h(\Gamma)}\langle (A\sigma_h t_e)\cdot \nu_e,\partial_{t_e}(\phi-I_h\phi)\rangle_{e}\\
=&\sum\limits_{e\in\mathcal{E}_h(\Gamma)}\langle \partial_{t_e}((A\sigma_h t_e)\cdot \nu_e),\phi-I_h\phi\rangle_{e}
\end{align*}
where the fact $(\phi-I_h\phi)$ vanishing at the boundary vertices \eqref{preserve-vertex value} is used. So
\begin{align*}
\sum\limits_{K\in\mathcal{T}_h}\langle A\sigma_h t,{\rm \bf{Curl}}(\phi-I_h\phi)\rangle _{\partial K}&=\sum\limits_{e\in\mathcal{E}_h(\Omega)}\langle \big[(A\sigma_h t_e)\cdot t_e\big]_e,\partial_{\nu_e}(\phi-I_h\phi)\rangle_{e}\\
&+\sum\limits_{e\in\mathcal{E}_h(\Gamma)}\langle (A\sigma_h t_e)\cdot t_e,\partial_{\nu_e}(\phi-I_h\phi)\rangle_{e}\\
&+\sum\limits_{e\in\mathcal{E}_h(\Gamma)}\langle \partial_{t_e}((A\sigma_h t_e)\cdot \nu_e),\phi-I_h\phi\rangle_{e}.
\end{align*}
Substituting it into \eqref{byparts}, we get
\begin{align*}
 (A\sigma_h,{\rm \bf{Curl} \, \bf{Curl} \,}(\phi-I_h\phi))
   =&\sum\limits_{K\in\mathcal{T}_h}({\rm curl \, curl \,}(A\sigma_h),\phi-I_h\phi)_K\\
+&\sum\limits_{e\in\mathcal{E}_h(\Omega)}\langle \big[(A\sigma_h t_e)\cdot t_e\big]_e,\partial_{\nu_e}(\phi-I_h\phi)\rangle_{e} \\
-&\sum\limits_{e\in\mathcal{E}_h(\Omega)}\langle \big[{\rm curl}(A\sigma_h)\cdot t_e\big]_e,\phi-I_h\phi \rangle_{e}\\
+&\sum\limits_{e\in\mathcal{E}_h(\Gamma)}\langle (A\sigma_h t_e)\cdot t_e,\partial_{\nu_e}(\phi-I_h\phi)\rangle_{e} \\
+&\sum\limits_{e\in\mathcal{E}_h(\Gamma)}\langle \partial_{t_e}((A\sigma_h t_e)\cdot \nu_e)-{\rm curl}(A\sigma_h)\cdot t_e,\phi-I_h\phi\rangle_{e}.
\end{align*}
Then applying the Cauchy-Schwarz inequality, the error estimate of the quasi-interpolation \eqref{inter-error-K}, \eqref{inter-error-e}, we have
\begin{align}\label{error-tilde-sigma}
&\|\tilde{\sigma}-\sigma_h\|_A^2=(A(\tilde \sigma-\sigma_h),{\rm \bf{Curl} \, \bf{Curl} \,}\phi)\nonumber\\
&\lesssim \left [\sum\limits_{K\in\mathcal{T}_h}h_K^4\|{\rm curl~curl}(A\sigma_h)\|_{0,K}^2+\sum\limits_{e\in\mathcal{E}_h}\left(h_e\|\mathcal{J}_{e,1}\|_{0,e}^2+h_e^3\|\mathcal{J}_{e,2}\|_{0,e}^2\right)\right ]^{\frac{1}{2}}|\phi|_2\\
&\lesssim \left [\sum\limits_{K\in\mathcal{T}_h}\eta_K^2(\sigma_h)+\sum\limits_{e\in\mathcal{E}_h}\eta_e^2(\sigma_h)\right ]^{\frac{1}{2}}\| {\rm \bf{Curl}\, \bf{Curl}\,} \phi\|_0.\nonumber
\end{align}
By \cite{CarstensenDolzmann1998}, the $\phi$ defined in \eqref{HD-error} satisfies  that
${\rm div}(\bf{Curl}~\bf{Curl}~\phi)=0$ and
$$
\int_{\Omega}{\rm tr}(\bf{Curl}~\bf{Curl}~\phi){\rm d}x=\int_{\Omega}{\rm tr}(\tilde \sigma-\sigma_h-\mathcal{C}\varepsilon(v)){\rm d}x=-\int_{\Omega}{\rm tr}(\mathcal{C}\varepsilon(v)){\rm d}x=0.
$$
Using Proposition~9.1.1 in \cite{BoffiBrezziFortin2013}, we get
$$
\| {\rm \bf{Curl}\, \bf{Curl}\,} \phi\|_0 \leq C \| {\rm \bf{Curl}\, \bf{Curl}\,} \phi\|_A,
$$
where the constant $C$ is independent of the Lam\'{e} constant $\lambda$. Combining this with
\eqref{HD-error-estimate}, \eqref{error-tilde-sigma}, we obtain
%
$$
\|\tilde\sigma-\sigma_h\|_A\lesssim \left [\sum\limits_{K\in\mathcal{T}_h}\eta_K^2(\sigma_h)+\sum\limits_{e\in\mathcal{E}_h}\eta_e^2(\sigma_h)\right ]^{\frac{1}{2}}.
$$
Together with the triangle inequality and perturbation result \eqref{oscf}, we get the desired error bound
\begin{align*}
\|\sigma-\sigma_h\|_A&\leq \|\sigma-\tilde \sigma\|_A + \|\tilde \sigma-\sigma_h\|_A\\
&\lesssim \left [\sum\limits_{K\in\mathcal{T}_h}\eta_K^2(\sigma_h)+\sum\limits_{e\in\mathcal{E}_h}\eta_e^2(\sigma_h)\right ]^{\frac{1}{2}} + {\rm osc}(f,\mathcal{T}_h).
\end{align*}
In summary, we obtain the following upper bound estimation.
\begin{theorem}[Reliability of the error estimator]
Let $(\sigma, u)$ be the solution of the mixed formulation \eqref{eqn1} and $(\sigma_h, u_h)$ be the solution of the mixed finite element method \eqref{mfem}. If the compliance tensor
$A$ is continuous, there exists positive constant $C_1$ depending only on the shape-regularity of the triangulation and the polynomial degree $k$ such that
\begin{equation}\label{eq:posterioriestimateupperbound}
\|\sigma-\sigma_h\|_{A}\leq C_1(\eta(\sigma_{h}, {\mathcal{T}_h}) + {\rm osc}(f,\mathcal{T}_h)).
\end{equation}
\end{theorem}

\medskip

\begin{remark}\label{rm:disA}
When $A$ is discontinuous, we can modify $\eta(\sigma_{h}, \mathcal{T}_h)$ as follows:
\begin{align*}
\eta^2(\sigma_h, \mathcal{T}_h):& =\sum_{K\in\mathcal{T}_h}h_K^4\|{\curl\curl}(A\sigma_h)\|_{0,K}^2 + \sum_{e\in\mathcal{E}_h}h_e\|[(A\sigma_h)t_e \cdot t_e]\|_{0,e}^2 \\
&+ \sum_{e\in\mathcal{E}_h}h_e^3\|\big[{\curl}(A\sigma_h)\cdot t_e-\partial_ {t_e}((A\sigma_h)t_e \cdot \nu_e)\big] \|_{0, e}^2.
\end{align*}
Compared to the case of continuous coefficient $A$, this estimator includes an additional term, the jump of $\partial_ {t_e}((A\sigma_h)t_e \cdot \nu_e)$ on all interior edges, owing to the discontinuity of the matrix $A$. Similarly, we can prove the reliability of the estimator
\[
\|\sigma-\sigma_h\|_{A}\lesssim \eta(\sigma_{h}, {\mathcal{T}_h}) + {\rm osc}(f,\mathcal{T}_h).
\]
\end{remark}
\begin{remark}
\rm
By Proposition~9.1.1 in \cite{BoffiBrezziFortin2013}, it holds
\begin{equation*}
\|\tau\|_0\lesssim \|\tau\|_A + \|\div\tau\|_{-1} \quad \forall~\tau\in\hat{\Sigma}
\end{equation*}
where $\hat\Sigma:=\{\tau\in \Sigma: (\mathrm{tr}\tau, 1)=0\}$ with $\mathrm{tr}$ being the trace operator of matrix.
Then we also have from \eqref{eq:posterioriestimateupperbound} and the fact that $\|f-f_h\|_{-1}\lesssim {\rm osc}(f,\mathcal T_h)$
\begin{align*}
\|\sigma-\sigma_h\|_{0}\lesssim &\|\sigma-\sigma_h\|_{A}+ \|\div(\sigma-\sigma_h)\|_{-1}\lesssim \eta(\sigma_{h}, {\mathcal{T}_h}) + {\rm osc}(f,\mathcal{T}_h).
\end{align*}
That is we can control the $L^2$ norm of the stress with constant independent of the Lam\'e constant $\lambda$. $\quad\quad\quad\quad\quad\quad\quad\quad\quad\quad\quad\quad\quad\quad\quad\quad\quad\quad\quad\quad\quad\quad\quad\quad\quad\quad\quad\quad\quad\quad\quad\;\Box$
\end{remark}


\subsection{Lower bound}
We shall follow Alonso \cite{Alonso1996} to prove the efficiency of the error estimator defined in \eqref{estimator-1}. Similar to \cite{Alonso1996}, we need the following lemma.
\begin{lemma}\label{Morgan-inter}
For any $K\in\mathcal{T}_h$, given $p_K\in L^2(K),~ q_e\in L^2(e),~r_e\in L^2(e),~e\in\partial K,$ there exists a unique $\psi_K\in P_{k+4}(K)$ satisfying that
\bq\label{psi}
\left\{\begin{array}{rll}
 \disp(\psi_K,v)=&(p_K,v)_K&\text{ for any } v\in P_{k-2}(K),\\
\disp\langle\psi_K, s\rangle_e=& \langle q_e,s\rangle_e& \text{ for any } s\in P_{k-1}(e),\\
\disp \langle \partial_{\nu}\psi_K,s  \rangle_e=&\langle r_e,s\rangle_e&\text{ for any } s\in P_k(e),\\
 \disp\partial^{\alpha}\psi_K(P)=&0& |\alpha|\leq 2,~~\text{ for any vertex}~ P\in~K,
\end{array}\right.
\eq
where $P_k(e)$ denotes the spaces of polynomial of degree less than or equal to $k$ on edge $e$.  Moreover it holds that
 \begin{align}\label{psi-estimate-1}
\|\psi_K\|_{0,K}^2\lesssim \|p_K\|_{0,K}^2+
\sum\limits_{e\in\partial K}\left (
h_e\|q_e\|_{0,e}^2+h_e^{3}\|r_e\|_{0,e}^2
\right ).
\end{align}
\end{lemma}
\begin{proof}
 Similar as in \cite{MorganScott1975}, such a function $\psi_K$ is determined uniquely by the above degrees of freedoms.  A standard homogeneity argument gives \eqref{psi-estimate-1}.
\end{proof}
\begin{theorem}[Efficiency of the error estimator]
Let $(\sigma, u)$ be the solution of the mixed formulation \eqref{eqn1} and $(\sigma_h, u_h)$ be the solution of the mixed finite element method \eqref{mfem}. If the compliance tensor
$A$ is continuous, there exists positive constant $C_2$ depending only on the shape-regularity of the triangulations and the polynomial degree $k$ such that
\begin{equation}\label{eq:posterioriestimatelowerbound}
C_2\eta(\sigma_{h}, \mathcal{T}_h)\leq \|\sigma-\sigma_{h}\|_{A}.
\end{equation}
\end{theorem}
\begin{proof}

The estimator $\eta^2(\sigma_{h}, {\mathcal{T}_h})$ can be rewritten as
\begin{align*}
\eta^2(\sigma_{h}, {\mathcal{T}_h})&  =
\disp\sum\limits_{K\in\mathcal{T}_h}\left({\rm curl \, curl \,}(A\sigma_h), h_K^4{\rm curl \, curl \,}(A\sigma_h)\right)_K\\
&+\sum\limits_{K\in\mathcal{T}_h}\sum\limits_{e\in\partial K}\langle (A\sigma_h) t_e\cdot t_e,\, h_e\mathcal{J}_{e,1}\rangle_e\\
&+\sum\limits_{K\in\mathcal{T}_h}\sum\limits_{e\in\partial K\bigcap\mathcal{E}_h(\Omega)}\langle {\rm curl\,} (A\sigma_h) \cdot t_e, \, h_e^3\mathcal{J}_{e,2}\rangle_e\\
&+\sum\limits_{K\in\mathcal{T}_h}\sum\limits_{e\in\partial K\bigcap\mathcal{E}_h(\Gamma)}\langle {\rm curl\,} (A\sigma_h) \cdot t_e-\partial_{t_e}
\left((A\sigma_h) t_e\cdot \nu_e\right),\, h_e^3\mathcal{J}_{e,2}\rangle_e.
\end{align*}
On each element $K\in\mathcal{T}_h$, we apply Lemma \ref{Morgan-inter} for $p_K=h_K^4{\rm curl \, curl \,}(A\sigma_h)|_K$, $q_e=-h_e^3\mathcal{J}_{e,2}$, $r_e=h_e\mathcal{J}_{e,1}$ for each edge $e\in\partial K$. Let $\psi|_K=\psi_K$, such a defined $\psi$ is in the high-order Argyris finite element space of degree $k+4$, hence $\psi\in H^2(\Omega)$. Using \eqref{psi-estimate-1}, it follows that
\begin{align}\label{psi-estimate-2}
\|\psi\|_{0,K}^2&\lesssim h_K^8\|{\rm curl \, curl \,}(A\sigma_h)\|_{0,K}^2+\sum\limits_{e\in\partial K}\left (\,
h_e^7\|\mathcal{J}_{e,2}\|_{0,e}^2+h_e^5\|\mathcal{J}_{e,1}\|_{0,e}^2
\right ).
\end{align}
This, in conjunction with \eqref{psi}, yields
\begin{align}
\eta^2(\sigma_{h}, {\mathcal{T}_h})&
=\disp\sum\limits_{K\in\mathcal{T}_h}\left({\rm curl \, curl \,}(A\sigma_h), \psi_K\right)_K\nonumber\\
&-\sum\limits_{K\in\mathcal{T}_h}\sum\limits_{e\in\partial K}\langle {\rm curl\,} (A\sigma_h) \cdot t_e,\, \psi_K\rangle_e\nonumber\\
&+\sum\limits_{K\in\mathcal{T}_h}\sum\limits_{e\in\partial K}\langle (A\sigma_h) t_e\cdot t_e,\, \partial_{\nu_e}\psi_K\rangle_e\label{eta-lowerbound-1}\\
&+\sum\limits_{K\in\mathcal{T}_h}\sum\limits_{e\in\partial K\bigcap\mathcal{E}_h(\Gamma)}\langle \partial_{t_e}
\left((A\sigma_h) t_e\cdot \nu_e\right),\, \psi_K\rangle_e.\nonumber
\end{align}
Since $(A\sigma_h) t_e\cdot \nu_e$ is continuous across the element edge $e$, $[A\sigma_h t_e\cdot \nu_e]_e=0$ on interior edges. Noting that $\psi\in H^2(\Omega)$ and vanishes at the mesh vertices,
\begin{align}
&\sum\limits_{K\in\mathcal{T}_h}\sum\limits_{e\in\partial K\bigcap\mathcal{E}_h(\Gamma)}\langle \partial_{t_e}
\left((A\sigma_h) t_e\cdot \nu_e\right),\, \psi_K\rangle_e\nonumber\\
&=-\sum\limits_{K\in\mathcal{T}_h}\sum\limits_{e\in\partial K\bigcap\mathcal{E}_h(\Gamma)}\langle
(A\sigma_h) t_e\cdot \nu_e,\, \partial_{t_e}\psi_K\rangle_e\label{boundary-normal-1}\\
&=-\sum\limits_{K\in\mathcal{T}_h}\sum\limits_{e\in\partial K}\langle
(A\sigma_h) t_e\cdot \nu_e,\, \partial_{t_e}\psi_K\rangle_e.\nonumber
\end{align}
Hence the last two terms of \eqref{eta-lowerbound-1} become
\begin{align}
&\sum\limits_{K\in\mathcal{T}_h}\sum\limits_{e\in\partial K}\langle (A\sigma_h) t_e\cdot t_e,\, \partial_{\nu_e}\psi_K\rangle_e\nonumber\\
+&\sum\limits_{K\in\mathcal{T}_h}\sum\limits_{e\in\partial K\bigcap\mathcal{E}_h(\Gamma)}\langle \partial_{t_e}
\left((A\sigma_h) t_e\cdot \nu_e\right),\, \psi_K\rangle_e\nonumber\\
=&\sum\limits_{K\in\mathcal{T}_h}\sum\limits_{e\in\partial K}\langle (A\sigma_h) t_e\cdot t_e,\, {\rm \bf{Curl}\,}\psi_K\cdot t_e\rangle_e\label{boundary-normal-2}\\
-&\sum\limits_{K\in\mathcal{T}_h}\sum\limits_{e\in\partial K}\langle
(A\sigma_h) t_e\cdot \nu_e,\, -{\rm\bf{Curl}\,}\psi_K\cdot \nu_e\rangle_e\nonumber\\
=&\sum\limits_{K\in\mathcal{T}_h}\sum\limits_{e\in\partial K}\langle (A\sigma_h) t_e,\, {\rm \bf{Curl}\,}\psi_K\rangle_e.\nonumber
\end{align}
Substituting \eqref{boundary-normal-2} into \eqref{eta-lowerbound-1} leads to
\begin{align*}
\eta^2(\sigma_{h}, {\mathcal{T}_h})=\disp\sum\limits_{K\in\mathcal{T}_h}\Big(\,\big({\rm curl \, curl \,}(A\sigma_h), \psi_K\big)_K
 &-\sum\limits_{e\in\partial K}\langle {\rm curl\,} (A\sigma_h) \cdot t_e,\, \psi_K\rangle_e\\
&  +\sum\limits_{e\in\partial K}\langle (A\sigma_h) t_e,\, {\rm \bf{Curl}\,}\psi_K\rangle_e\,\Big)
\end{align*}
Integrating the first term by parts twice,
\begin{align*}
\eta^2(\sigma_{h}, {\mathcal{T}_h})&=\sum\limits_{K\in\mathcal{T}_h}\left( A\sigma_h,\, {\rm \bf{Curl} \, \bf{Curl} \,}\psi_K\right )_K\\
&=\sum\limits_{K\in\mathcal{T}_h}\left (A(\sigma_h-\sigma),{\rm \bf{Curl} \, \bf{Curl} \,}\psi_K \right )_K\\
&\lesssim\|\sigma-\sigma_h\|_A\left(\sum\limits_{K\in\mathcal{T}_h}h_K^{-4}\|\psi\|_{0,K}^2\right)^{\frac{1}{2}},
\end{align*}
where  ${\rm \bf{Curl}\, \bf{Curl}}~\psi\in\Sigma$ and the inverse inequality are used. By \eqref{psi-estimate-2},
\begin{align*}
\sum\limits_{K\in\mathcal{T}_h}h_K^{-4}\|\psi\|_{0,K}^2&\lesssim\sum\limits_{K\in\mathcal{T}_h}h_K^4\|{\rm curl \, curl \,}(A\sigma_h)\|_{0,K}^2+\sum\limits_{e\in\mathcal{E}_h}\left( h_e\|\mathcal{J}_{e,1}\|_{0,e}^2+h_e^3\|\mathcal{J}_{e,2}\|_{0,e}^2\right)\\
&\widehat{=}\eta^2(\sigma_h,\mathcal{T}_h).
\end{align*}
Combining the above two inequalities, we have that
\begin{align*}
\eta(\sigma_{h}, {\mathcal{T}_h})\lesssim\|\sigma-\sigma_h\|_A.
\end{align*}
\end{proof}
\begin{remark}\rm
For discontinuous $A$ and the modified error estimator in Remark \ref{rm:disA}, efficiency can be also proved using a similar argument.
\end{remark}
\section{A posteriori error estimation for mixed boundary problems}

The {\it a posteriori} error estimation for the linear elasticity problems with the homogeneous Dirichlet boundary condition can be generalized to problems with mixed boundary conditions. In this section, we will discuss the following linear elasticity problems with mixed boundary conditions. Let $\Omega\subset\mathbb{R}^2$ be a bounded polygonal domain with boundary $\Gamma:=\partial\Omega=\Gamma_D\cup\Gamma_N$, $\Gamma_D\cap\Gamma_N=\emptyset$, $\Gamma_N\neq\emptyset$. Given data $f\in L^2(\Omega; \mathbb{R}^2)$, $u_D\in H^1(\Omega; \mathbb{R}^2)$, and $g\in L^2(\Gamma_N; \mathbb{R}^2)$, seek the solution $(\sigma,u)\in\Sigma_g\times V$, such that
\begin{equation}\label{eqn1-MB}
\left\{ \ad{
  (A\sigma,\tau)+({\rm div}\tau,u)&= \langle u_D, \tau\nu\rangle_{\Gamma_D}  && \hbox{for all \ } \tau\in\Sigma_0,\\
   ({\rm div}\sigma,v)&= (f,v) &\qquad& \hbox{for all \ } v\in V, }
   \right.
\end{equation}
where
\begin{align*}
\Sigma_0&:=\big\{\sigma\in H({\rm div},\Omega;\mathbb{S})|\int_{\Gamma_N}\psi\cdot(\sigma\nu)ds=0,~{\rm for~all}~\psi\in\mathcal{D}(\Gamma_N;\mathbb{R}^2)\big\},\\
\Sigma_g&:=\big\{\sigma\in H({\rm div},\Omega;\mathbb{S})|\int_{\Gamma_N}\psi\cdot(\sigma\nu)ds=\int_{\Gamma_N}\psi\cdot gds,~{\rm for~all}~\psi\in\mathcal{D}(\Gamma_N;\mathbb{R}^2)\big\},
\end{align*}
where $\mathcal{D}$ denotes the space of test functions. Let $\Sigma_{0,h}:=\Sigma_0\bigcap\Sigma_h,~~~~\Sigma_{g,h}:=\Sigma_g\bigcap \Sigma_h, $ the mixed finite element method seeks $(\sigma_h,u_h)\in\Sigma_{g,h}\times V_h$, such that
\begin{equation}\label{Meqn1-MB}
\left\{ \ad{
  (A\sigma_h,\tau_h)+({\rm div}\tau_h,u_h)&= \langle u_D, \tau_h\nu\rangle_{\Gamma_D}  && \hbox{for all \ } \tau_h\in\Sigma_{0,h},\\
   ({\rm div}\sigma_h,v_h)&= (f,v_h) &\qquad& \hbox{for all \ } v_h\in V_h. }
   \right.
\end{equation}

We modify the {\it a posterior} error estimator defined in Section 3 as the following:
$$
 \eta^2(\sigma_{h}, {\mathcal{T}_h}):=\sum\limits_{K\in\mathcal{T}_h}\eta_K^2(\sigma_h)+\sum\limits_{e\in\mathcal{E}_h}\eta_e^2(\sigma_h)
  $$
where
$$
\eta_K^2(\sigma_h) :=h_K^4\|{\rm curl \, curl \,}(A\sigma_h)\|_{0,K}^2,~~~
\eta_e^2(\sigma_h):=h_e\|\mathcal{J}_{e,1}\|_{0,e}^2+h_e^3\|\mathcal{J}_{e,2}\|_{0,e}^2,
$$
$$
\begin{array}{ll}
 \disp \mathcal{J}_{e,1}:=&\left\{\begin{array}{ll}
 \disp\Big[(A\sigma_h)t_e\cdot t_e\Big]_e&{\rm if~} e\in\mathcal{E}_h(\Omega)\\
 \disp\Big((A\sigma_h)t_e\cdot t_e - \partial_{t_e}(u_D\cdot t_e)\Big)|_e\quad\quad\quad\quad\quad\ & {\rm if~} e\in\mathcal{E}_h(\Gamma_D)\\
 \disp\Big((A\sigma_h)t_e\cdot t_e\Big)|_e&{\rm if~} e\in\mathcal{E}_h(\Gamma_N)
 \end{array}\right. \\
 \\
 \disp \mathcal{J}_{e,2}:=&\left\{\begin{array}{lr}
\disp \Big[{\rm curl}(A\sigma_h)\cdot t_e\Big]_e&~~~~ {\rm if~} e\in\mathcal{E}_h(\Omega)\\
\disp \Big({\rm curl}(A\sigma_h)\cdot t_e+\partial_{t_et_e}(u_D\cdot\nu)-\partial _{t_e}\big((A\sigma_h)t_e\cdot \nu_e\big)\Big)|_e&~~~~ {\rm if~} e\in\mathcal{E}_h(\Gamma_D)\\
\disp \Big({\rm curl}(A\sigma_h)\cdot t_e-\partial _{t_e}\big((A\sigma_h)t_e\cdot \nu_e\big)\Big)|_e&~~~~ {\rm if~} e\in\mathcal{E}_h(\Gamma_N)
 \end{array}\right.
 \end{array}
$$
where $\mathcal E_{h}(\Gamma_D)$, $\mathcal E_{h}(\Gamma_N)$ are the collection of element edges for Dirichlet boundary  and Neumann boundary respectively.

Similar to Section 3, we can prove the reliability and efficiency of this {\it a posteriori} error estimator.
\begin{theorem}[Reliability and efficiency of the error estimator]
Let $(\sigma, u)$ be the solution of the mixed formulation \eqref{eqn1-MB} and $(\sigma_h, u_h)$ be the solution of the mixed finite element method \eqref{Meqn1-MB}. If the compliance tensor
$A$ is continuous, there exist positive constant $C_3$ and $ C_4$ depending only on the shape-regularity of the triangulation and the polynomial degree $k$ such that
\begin{equation}\label{eq:posterioriestimateupperbound-MB}
\|\sigma-\sigma_h\|_{A}\leq C_3\Big(\eta(\sigma_{h}, {\mathcal{T}_h}) + {\rm osc}(f,\mathcal{T}_h)+{\rm osc}(g,\mathcal{E}_h(\Gamma_N))\Big),
\end{equation}
and
\begin{equation}\label{eq:posterioriestimatelowerbound-MB}
C_4\eta(\sigma_{h}, \mathcal{T}_h)\leq \|\sigma-\sigma_{h}\|_{A}+{\rm osc}(u_D,\mathcal{E}_h(\Gamma_D)),
\end{equation}
where the data oscillations for the Dirichlet boundary $u_D$ and the Neumann boundary condition $g$ are defined as
\begin{align*}
{\rm osc}(g,\mathcal{E}_h(\Gamma_N))^2:=&\sum\limits_{e\in\mathcal{E}_h(\Gamma_N)}h_e\|g-g_h\|_{0,e}^2\\
{\rm osc}(u_D,\mathcal{E}_h(\Gamma_D))^2:=&\sum\limits_{e\in\mathcal{E}_h(\Gamma_D)}h_e\|\partial_{t_e}(u_D\cdot t_e)-\partial_{t_e}(u_{D,h}\cdot t_e)\|_{0,e}^2\\
+&\sum\limits_{e\in\mathcal{E}_h(\Gamma_D)}h_e^3\|\partial_{t_et_e}(u_D\cdot \nu_e)-\partial_{t_et_e}(u_{D,h}\cdot \nu_e)\|_{0,e}^2,
\end{align*}
$g_h$ is the piecewise $L^2$ projection of $g$ onto $P_k(\mathcal{E}_h(\Gamma_N),\mathbb{R}^2)$ and
$u_{D,h}$ is the piecewise $L^2$ projection of $u_D$ onto $P_k(\mathcal{E}_h(\Gamma_D),\mathbb{R}^2)$.
\end{theorem}

\section{A Posteriori Error Estimation for Displacement}

In this section, we shall discuss the {\it a posteriori} error estimate for a superconvergent postprocessed displacement recently constructed in \cite{ChenHuHuang2016}. The key points of the theoretical analysis involve the discrete inf-sup condition and the norm equivalence on $H^1(\mathcal{T}_h; \mathbb{R}^2)$ developed in \cite{ChenHuHuang2016}, and the {\it a posteriori} error estimates \eqref{eq:posterioriestimateupperbound} and \eqref{eq:posterioriestimatelowerbound}.
Here the broken space
\[
H^1(\mathcal{T}_h; \mathbb{R}^2):=\left\{v\in L^2(\Omega; \mathbb{R}^2): v|_K\in H^1(K; \mathbb{R}^2)\quad \forall\,K\in\mathcal
{T}_h\right\}.
\]
For any $v\in H^1(\mathcal{T}_h; \mathbb{R}^2)$,
define mesh dependent norm
\begin{align*}
|v|_{1,h}^2 &:= \|\varepsilon_h (v)\|_0^2 + \sum_{e\in \mathcal E_h} h_e^{-1}\|[v]\|_{0,e}^2.
\end{align*}

We first recall the superconvergent postprocessed displacement from $(\sigma_h,  u_h)$ developed in \cite{ChenHuHuang2016}.
To this end, let
\[
 V_{h}^{\ast}:=\left\{v\in  L^2(\Omega; \mathbb{R}^2):  v|_K\in  P_{k+1}(K; \mathbb{R}^2)\quad \forall\,K\in\mathcal
{T}_h\right\}.
\]
Then a postprocessed displacement is defined as follows \cite{ChenHuHuang2016, LovadinaStenberg2006, BrambleXu1989}:
Find $ u_h^{\ast}\in V_{h}^{\ast}$ such that
\begin{equation}\label{postprocess1}
 (u_h^{\ast}, v)_K= (u_h, v)_K \quad \forall~ v\in P_{k-1}(K; \mathbb{R}^2),
\end{equation}
\begin{equation}\label{postprocess2}
(\varepsilon(u_h^{\ast}), \varepsilon(w))_K=(A\sigma_h, \varepsilon(w))_K \quad \forall~ w\in (I- Q_h) V_{h}^{\ast}|_K,
\end{equation}
for any $K\in\mathcal{T}_h$.

We recall the following two useful results \cite{ChenHuHuang2016}:
the discrete inf-sup condition
\begin{equation}\label{eq:infsup22}
|v_h|_{1,h}\lesssim \sup_{0\neq\tau_h \in \Sigma_{h}} \frac{(\div\tau_h, v_h)}{\|\tau_h\|_{0}} \quad \forall~v_h\in V_{h},
\end{equation}
and norm equivalence
\begin{equation}\label{eq:temp4post}
| v- Q_h v|_{1,h} \eqsim \|\varepsilon_h ( v- Q_h v)\|_0 \quad \forall~v\in H^1(\mathcal{T}_h; \mathbb{R}^2).
\end{equation}

\begin{theorem}
Let $(\sigma, u)$ be the solution of the mixed formulation \eqref{eqn1}, $(\sigma_h, u_h)$ be the solution of the mixed finite element method \eqref{mfem}, and $u_h^{\ast}$ be the postprocessed displacement defined by \eqref{postprocess1}-\eqref{postprocess2}.
Then we have
\begin{equation}\label{eq:postdispalcementupperbound}
\|\sigma-\sigma_h\|_{A} + |u-u_h^{\ast}|_{1,h}\lesssim \eta(\sigma_{h}, {\mathcal{T}_h}) + \|A\sigma_h-\varepsilon_h(u_h^{\ast})\|_{0} + {\rm osc}(f,\mathcal{T}_h).
\end{equation}
\begin{equation}\label{eq:postdispalcementlowerbound}
\eta(\sigma_{h}, \mathcal{T}_h) + \|A\sigma_h-\varepsilon_h(u_h^{\ast})\|_{0} \lesssim \|\sigma-\sigma_h\|_{A} + |u-u_h^{\ast}|_{1,h}.
\end{equation}
\end{theorem}
\begin{proof}
Using the discrete inf-sup condition~\eqref{eq:infsup22} with $v_h=Q_h(u-u_h^{\ast})$, \eqref{postprocess1}, the first equations of \eqref{eqn1} and \eqref{mfem}, we get
\begin{align*}
|Q_h(u-u_h^{\ast})|_{1,h}\lesssim & \sup_{0\neq\tau_h \in \Sigma_{h}} \frac{(\div\tau_h, Q_h(u-u_h^{\ast}))}{\|\tau_h\|_{0}}=\sup_{0\neq\tau_h \in \Sigma_{h}} \frac{(\div\tau_h, u-u_h)}{\|\tau_h\|_{0}} \\
=&\sup_{0\neq\tau_h \in \Sigma_{h}} \frac{(A(\sigma-\sigma_h), \tau_h)}{\|\tau_h\|_{0}} \leq \|A(\sigma-\sigma_h)\|_0.
\end{align*}
Choosing $v=u-u_h^{\ast}$ in \eqref{eq:temp4post},
\begin{align*}
| v- Q_h v|_{1,h} \eqsim &\|\varepsilon_h ( v- Q_h v)\|_0 \leq \|\varepsilon_h (u-u_h^{\ast})\|_0 + |Q_h(u-u_h^{\ast})|_{1,h} \\
=&  \|A\sigma - \varepsilon_h (u_h^{\ast})\|_0 + |Q_h(u-u_h^{\ast})|_{1,h} \\
\lesssim &  \|A\sigma_h - \varepsilon_h (u_h^{\ast})\|_0 + \|A(\sigma-\sigma_h)\|_0.
\end{align*}
Then it follows from the last two inequalities that
\[
|u-u_h^{\ast}|_{1,h}\lesssim \|A\sigma_h - \varepsilon_h (u_h^{\ast})\|_0 + \|A(\sigma-\sigma_h)\|_0,
\]
which combined with \eqref{eq:posterioriestimateupperbound} implies \eqref{eq:postdispalcementupperbound}.

Next we prove the efficiency \eqref{eq:postdispalcementlowerbound}. By the triangle inequality,
\begin{align*}
\|A\sigma_h - \varepsilon_h (u_h^{\ast})\|_0\leq & \|A(\sigma-\sigma_h)\|_0+ \|A\sigma - \varepsilon_h (u_h^{\ast})\|_0 \\
=& \|A(\sigma-\sigma_h)\|_0 + \|\varepsilon_h (u-u_h^{\ast})\|_0 \\
\lesssim & \|\sigma-\sigma_h\|_A + |u-u_h^{\ast}|_{1,h}.
\end{align*}
Therefore we can end the proof by using \eqref{eq:posterioriestimatelowerbound}.
\end{proof}


\section{Numerical experiments}

We will testify the {\it a posteriori} error estimator by some numerical examples in this section.

In the first example, let $\Omega=(0, 1)^2$, $k=3$, $\mu=1$,
the right-hand side
\[
f(x,y)=\pi^3\left(
\begin{array}{c}
-\sin(2\pi y)(2\cos(2\pi x)-1) \\
\sin(2\pi x)(2\cos(2\pi y)-1)
\end{array}
\right),
\]
and the exact solution \cite[Section~5.2]{CarstensenGedicke2016}
\[
u(x,y)=\frac{\pi}{2}\left(
\begin{array}{c}
\sin^2(\pi x)\sin(2\pi y) \\
-\sin^2(\pi y)\sin(2\pi x)
\end{array}
\right).
\]
We subdivide $\Omega$ by a uniform triangular mesh.
The {\it a priori} and {\it a posteriori} error estimates for $\lambda=10$ and $\lambda=10000$ are listed in Tables~\ref{table:lambda10}-\ref{table:lambda10000},
from which we can see that the convergence rates of $\|\sigma-\sigma_h\|_A$, $\|\nabla_h(u-u_h^{\ast})\|_{0}$, $\eta(\sigma_h, \mathcal{T}_h)$ and $\|A\sigma_h - \varepsilon_h(u_h^{\ast})\|_0$ are all $O(h^{4})$.
Hence the {\it a posteriori} error estimators $\eta(\sigma_h, \mathcal{T}_h)$ and $\eta(\sigma_h, \mathcal{T}_h)+\|A\sigma_h - \varepsilon_h(u_h^{\ast})\|_0$ are both uniformly reliable and efficient
with respect to the mesh size $h$ and $\lambda$ for smooth solutions.
\begin{table}[htbp]
  \centering
  \caption{Numerical errors for the first example when $\lambda=10$}\label{table:lambda10}
\resizebox{\textwidth}{!}{ %
  \begin{tabular}{|c|c|c|c|c|c|c|c|c|}
     \hline
     $h$ & $\|\sigma-\sigma_h\|_A$ & order & $\|\nabla_h(u-u_h^{\ast})\|_{0}$ & order & $\eta(\sigma_h, \mathcal{T}_h)$ & order & $\|A\sigma_h - \varepsilon_h(u_h^{\ast})\|_0$ & order \\
\hline $2^{-1}$ & 6.6998E-01 & $-$ & 7.9544E-01 & $-$ & 1.6615E+01 & $-$ & 4.0073E-02 & $-$ \\
\hline $2^{-2}$ & 5.2451E-02 & 3.68 & 6.0585E-02 & 3.71 & 1.3585E+00 & 3.61 & 9.3899E-03 & 2.09 \\
\hline $2^{-3}$ & 3.6139E-03 & 3.86 & 4.5839E-03 & 3.72 & 1.0918E-01 & 3.64 & 7.1387E-04 & 3.72 \\
\hline $2^{-4}$ & 2.2714E-04 & 3.99 & 3.0676E-04 & 3.90 & 7.4510E-03 & 3.87 & 4.5925E-05 & 3.96 \\
\hline $2^{-5}$ & 1.4193E-05 & 4.00 & 1.9600E-05 & 3.97 & 4.7919E-04 & 3.96 & 2.8824E-06 & 3.99 \\
\hline $2^{-6}$ & 8.8742E-07 & 4.00 & 1.2347E-06 & 3.99 & 3.0263E-05 & 3.99 & 1.8040E-07 & 4.00 \\
\hline $2^{-7}$ & 5.5567E-08 & 4.00 & 7.7435E-08 & 3.99 & 1.8992E-06 & 3.99 & 1.1306E-08 & 4.00 \\
 \hline
   \end{tabular}
 }%
\end{table}
\begin{table}[htbp]
  \centering
  \caption{Numerical errors for the first example when $\lambda=10000$}\label{table:lambda10000}
\resizebox{\textwidth}{!}{ %
  \begin{tabular}{|c|c|c|c|c|c|c|c|c|}
     \hline
     $h$ & $\|\sigma-\sigma_h\|_A$ & order & $\|\nabla_h(u-u_h^{\ast})\|_{0}$ & order & $\eta(\sigma_h, \mathcal{T}_h)$ & order & $\|A\sigma_h - \varepsilon_h(u_h^{\ast})\|_0$ & order \\
\hline $2^{-1}$ & 6.6096E-01 & $-$ & 7.7905E-01 & $-$ & 1.6050E+01 & $-$ & 4.3292E-02 & $-$ \\
\hline $2^{-2}$ & 5.1630E-02 & 3.68 & 5.8762E-02 & 3.73 & 1.3066E+00 & 3.62 & 9.0182E-03 & 2.26 \\
\hline $2^{-3}$ & 3.5430E-03 & 3.87 & 4.3977E-03 & 3.74 & 1.0508E-01 & 3.64 & 6.8780E-04 & 3.71 \\
\hline $2^{-4}$ & 2.2220E-04 & 4.00 & 2.9277E-04 & 3.91 & 7.1542E-03 & 3.88 & 4.4330E-05 & 3.96 \\
\hline $2^{-5}$ & 1.3873E-05 & 4.00 & 1.8668E-05 & 3.97 & 4.5947E-04 & 3.96 & 2.7853E-06 & 3.99 \\
\hline $2^{-6}$ & 8.6708E-07 & 4.00 & 1.1751E-06 & 3.99 & 2.8998E-05 & 3.99 & 1.7442E-07 & 4.00 \\
\hline $2^{-7}$ & 5.4210E-08 & 4.00 & 7.3695E-08 & 4.00 & 1.8195E-06 & 3.99 & 1.0922E-08 & 4.00 \\
 \hline
   \end{tabular}
 }%
\end{table}

Next we use the {\it a posteriori} error estimator $\eta(\sigma_h, \mathcal{T}_h)$ to design an adaptive mixed finite element method, i.e. Algorithm~\ref{alg:amfem}.
The approximate block factorization preconditioner with GMRES \cite{ChenHuHuang2016} is adopted in the SOLVE part of Algorithm~\ref{alg:amfem}, which is verified to be highly efficient and robust even on adaptive meshes by our numerical examples.
\begin{algorithm}
\label{alg:amfem}
\caption{Adaptive algorithm for the mixed finite element method~\eqref{mfem}.}
Given a parameter $0<\vartheta<1$ and an initial mesh
$\mathcal{T}_{0}$. Set $m:=0$.
\begin{enumerate}[1.]
    \item \textbf{SOLVE}: Solve the mixed finite element method \eqref{mfem} on $\mathcal{T}_{m}$ for
    the discrete solution $(\sigma_{m},u_{m})\in\Sigma_{m}\times
V_{m}$.
    \item \textbf{ESTIMATE}: Compute the error indicator $\eta^2(\sigma_{m}, {\mathcal{T}_m})$ piecewise.
    \item \textbf{MARK}: Mark a set $\mathcal{S}_{m}\subset\mathcal{T}_{m}$ with minimal cardinality by  D\"{o}rfler marking such that
    \[
        \eta^2(\sigma_{m}, {\mathcal{S}_m})\geq\vartheta \eta^2(\sigma_{m}, {\mathcal{T}_m}).
            \]
    \item \textbf{REFINE}: Refine each triangle $K$ with at least one edge in $\mathcal{S}_{m}$ by the newest vertex bisection to get
    $\mathcal{T}_{m+1}$.
    \item Set $m:=m+1$ and go to Step 1.
\end{enumerate}
\end{algorithm}

Now we construct a problem with singularity in the solution to test Algorithm~\ref{alg:amfem}.
Set L-shaped domain $\Omega=(-1,1)\times(-1,1)\backslash
[0,1)\times(-1,0]$.
Let
\[
\Phi_1(\theta)=\left(
\begin{array}{c}
\big((z+2)(\lambda+\mu)+4\mu\big)\sin(z\theta) - z(\lambda+\mu)\sin((z-2)\theta) \\
    z(\lambda+\mu)\big(\cos(z\theta)-\cos((z-2)\theta)\big)
\end{array}
\right),
\]
\[
\Phi_2(\theta)=\left(
\begin{array}{c}
z(\lambda+\mu)\big(\cos((z-2)\theta)-\cos(z\theta)\big) \\
    -\big((2-z)(\lambda+\mu)+4\mu\big)\sin(z\theta) - z(\lambda+\mu)\sin((z-2)\theta)
\end{array}
\right),
\]
\begin{align*}
\Phi(\theta) = &\Big(z(\lambda+\mu)\sin((z-2)\omega) + \big((2-z)(\lambda+\mu)+4\mu\big)\sin(z\omega)\Big)\Phi_1(\theta) \\
&- z(\lambda+\mu)\big(\cos((z-2)\omega)-\cos(z\omega)\big)\Phi_2(\theta),
\end{align*}
where $z\in(0, 1)$ is a real root of $(\lambda+3\mu)^2\sin^2(z\omega)=(\lambda+\mu)^2z^2\sin^2\omega$ with $\omega=3\pi/2$.
The exact singular solution in polar coordinates is taken as \cite[Section~4.6]{Grisvard1992}
\[
u(r,\theta)=\frac{1}{(\lambda+\mu)^2}(r^2\cos^2\theta-1)(r^2\sin^2\theta-1)r^{z}\Phi(\theta).
\]
It can be computed that $z=0.561586549334359$ for $\lambda=10$, and $z=0.544505718203590$ for $\lambda=10000$.
We also take $k=3$ and $\mu=1$.

Some meshes generated by Algorithm~\ref{alg:amfem} for different bulk parameter $\vartheta$ and Lam\'e constant $\lambda$
are shown in Figure~\ref{fig:ldomainIter}, where $\#\textrm{dofs}$ is the number of degrees of freedom.
The adaptive Algorithm~\ref{alg:amfem} captures the singularity of the exact solution on the corner $(0, 0)$ very well.
The histories of the adaptive Algorithm~\ref{alg:amfem} for $\vartheta=0.1, 0.2$ and $\lambda=10, 10000$ are presented
in Figures~\ref{fig:errorex2lambda10}-\ref{fig:errorex2lambda10000}.
We can see from Figures~\ref{fig:errorex2lambda10}-\ref{fig:errorex2lambda10000} that
the convergence rates of errors $\|\sigma-\sigma_h\|_A$ and $\eta(\sigma_h, \mathcal{T}_h)$ are both $O((\#\textrm{dofs})^{-2})$
no matter $\lambda=10$ or $\lambda=10000$, which demonstrates the theoretical results. For uniform grid,  $\#(\textrm{dofs})^{-2}\cong h^4$, this means that the errors $\|\sigma-\sigma_h\|_A$ and $\eta(\sigma_h, \mathcal{T}_h)$ converge with an optimal rate.
\begin{figure}[htbp]
\subfigure[Initial mesh]{
\begin{minipage}[t]{0.5\linewidth}
\centering
\includegraphics[scale=0.58]{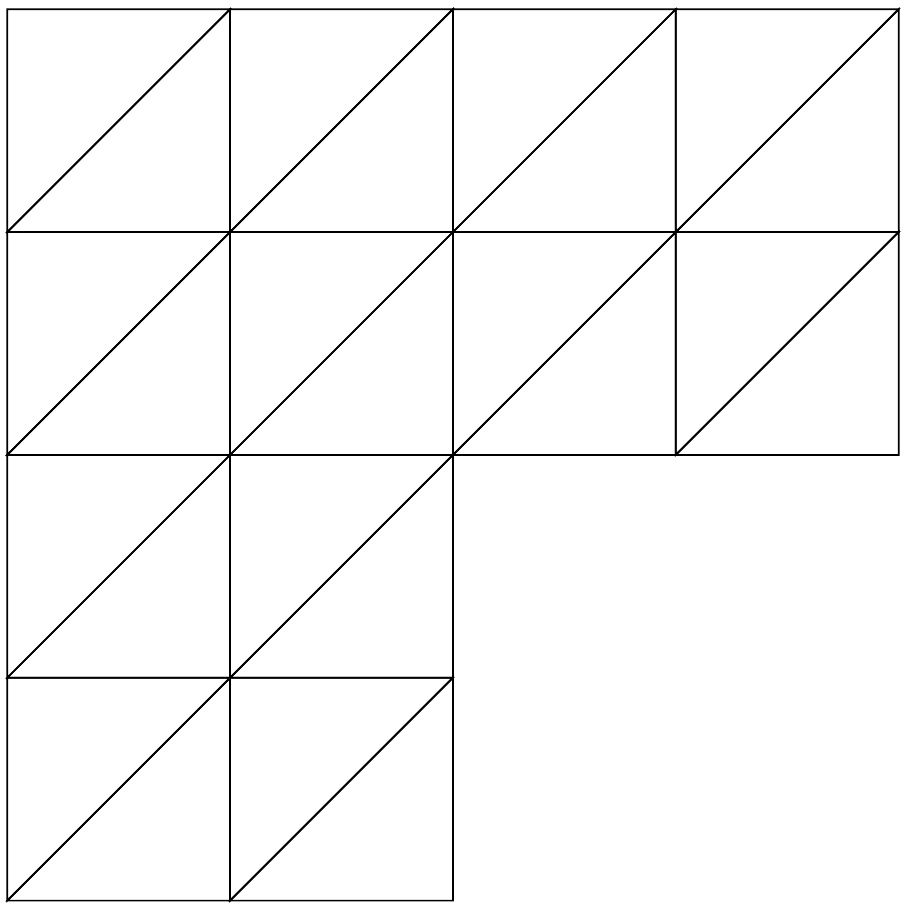}
\end{minipage}}
\subfigure[$\#\textrm{dofs}=198098, \theta=0.1, \lambda=10$]
{\begin{minipage}[t]{0.5\linewidth}
\centering
\includegraphics[scale=0.58]{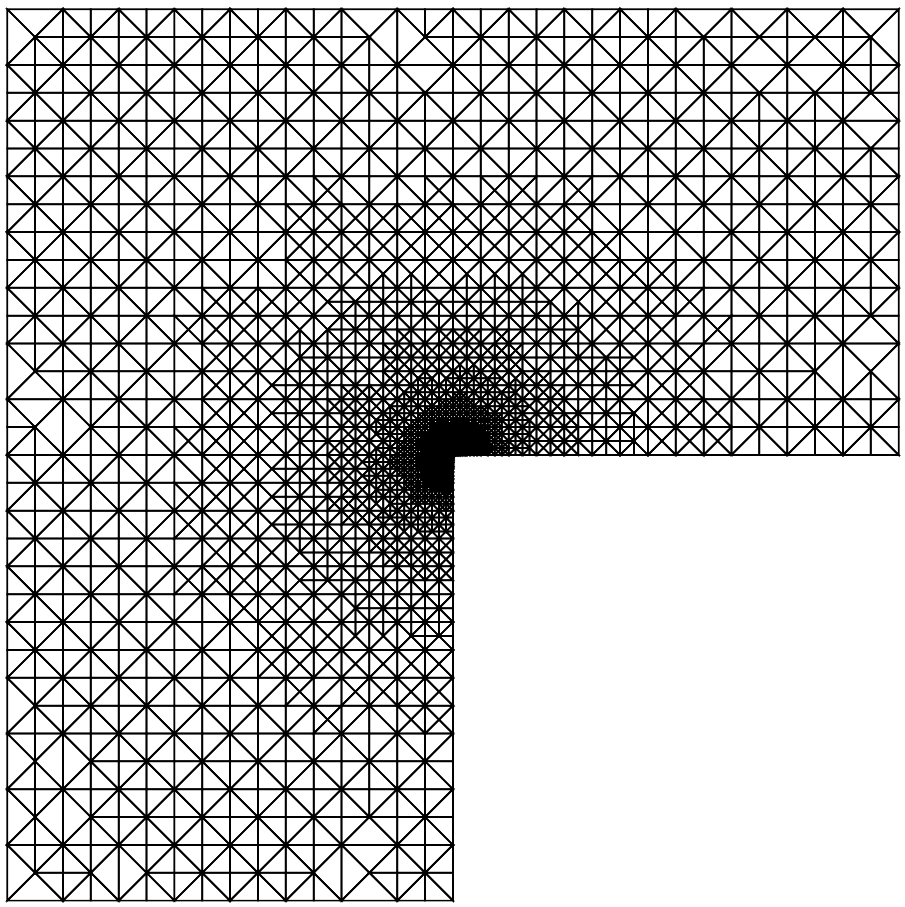}
\end{minipage}}
\vskip 0.1cm
\subfigure[$\#\textrm{dofs}=129624, \theta=0.2, \lambda=10$]
{\begin{minipage}[t]{0.5\linewidth}
\centering
\includegraphics[scale=0.58]{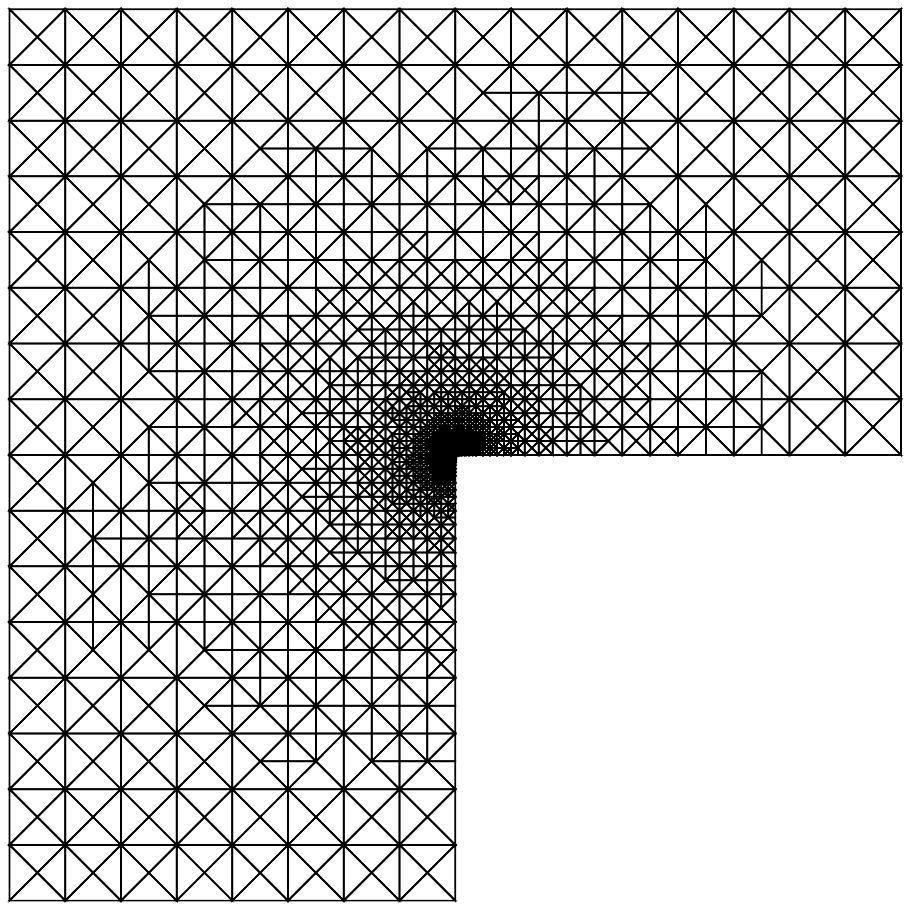}
\end{minipage}}
\subfigure[$\#\textrm{dofs}=138323, \theta=0.2, \lambda=10000$]
{\begin{minipage}[t]{0.5\linewidth}
\centering
\includegraphics[scale=0.58]{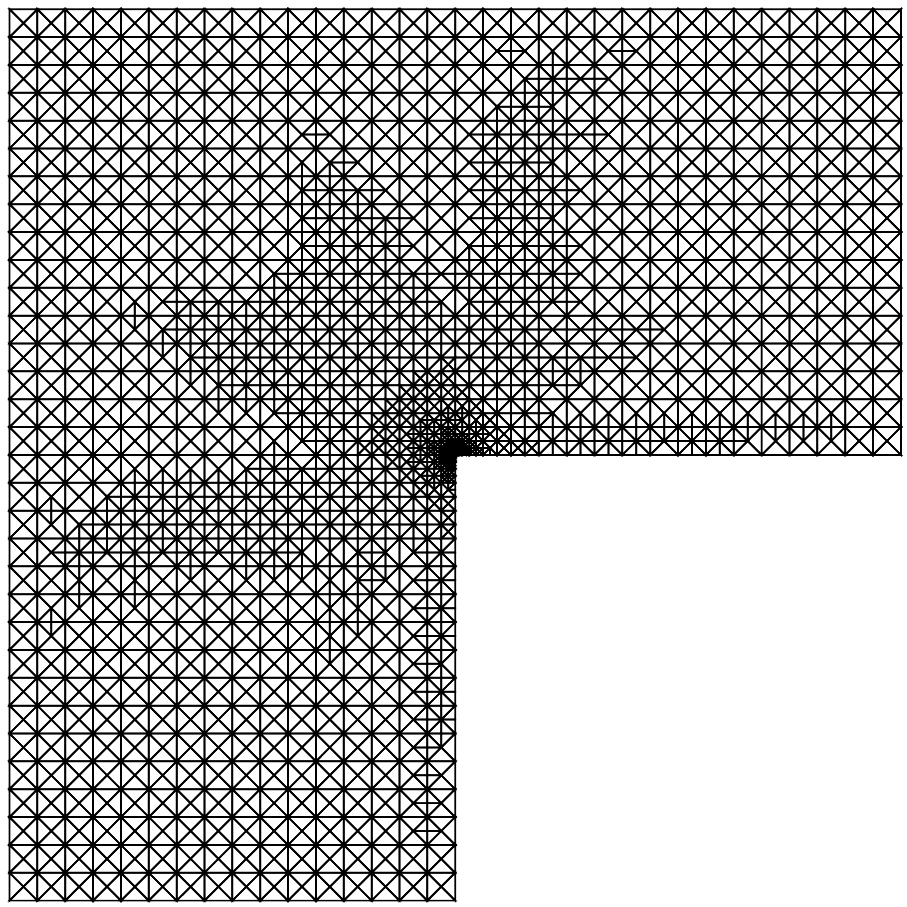}
\end{minipage}}
\vskip -0.2cm
\caption{Meshes generated in Algorithm~\ref{alg:amfem} with different $\theta$ and $\lambda$ for Example 2}
\label{fig:ldomainIter}
\end{figure}
\begin{figure}[htbp]
  \centering
    \includegraphics[scale=0.8]{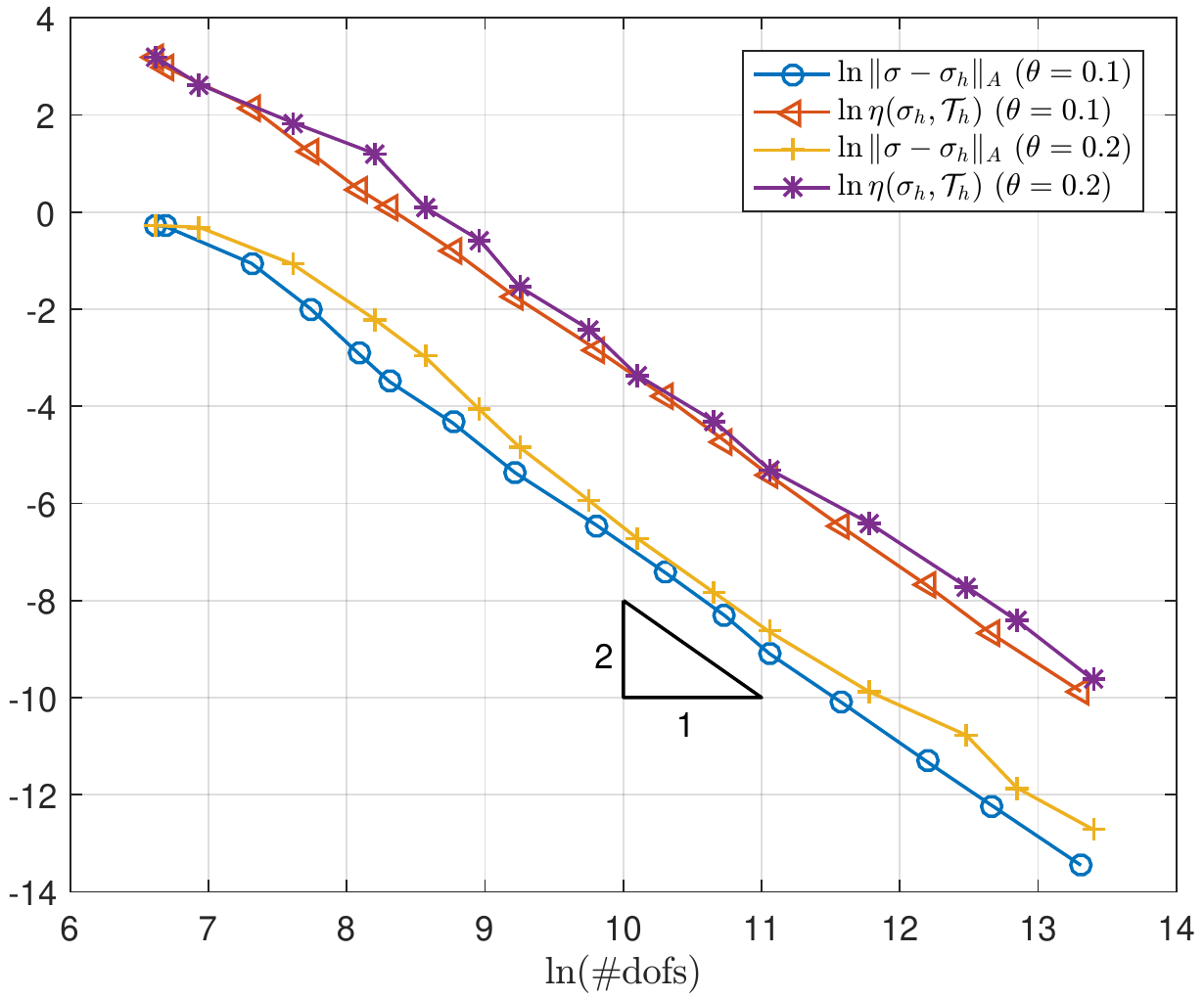}
  \caption{Errors $\|\sigma-\sigma_h\|_A$ and $\eta(\sigma_h, \mathcal{T}_h)$ vs $\#$dofs in
$\ln$-$\ln$ scale for Example 2 with $\lambda=10$.} \label{fig:errorex2lambda10}
\end{figure}
\begin{figure}[htbp]
  \centering
    \includegraphics[scale=0.8]{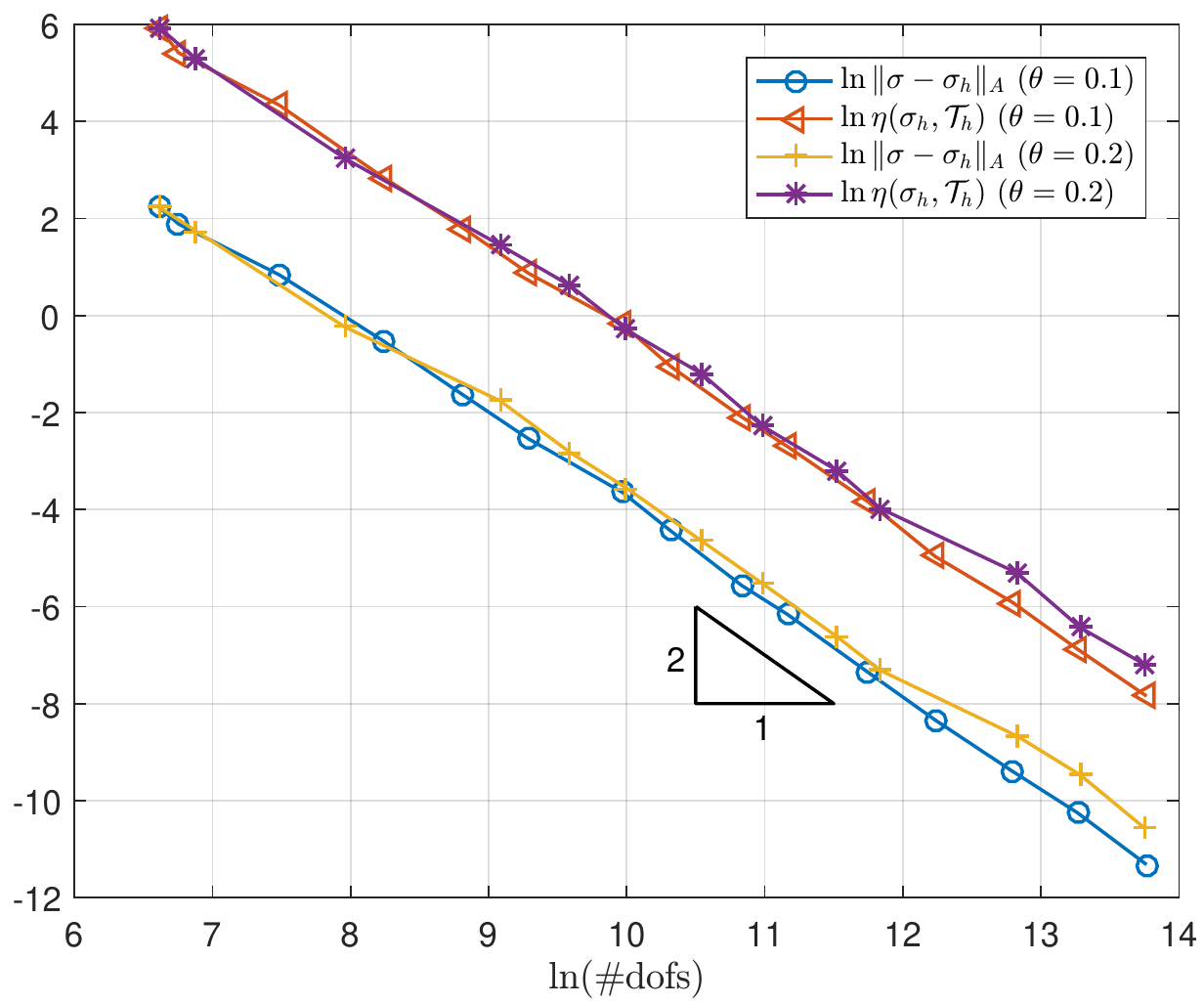}
  \caption{Errors $\|\sigma-\sigma_h\|_A$ and $\eta(\sigma_h, \mathcal{T}_h)$ vs $\#$dofs in
$\ln$-$\ln$ scale for Example 2 with $\lambda=10000$.} \label{fig:errorex2lambda10000}
\end{figure}

The third example considers the L-shape benchmark problem with general boundary conditions testified in \cite[section~5.3]{CarstensenGedicke2016} on the rotated L-shaped domain with the initial mesh as depicted in Figure~\ref{fig:initialdomainex3}. We impose the Neumann boundary
condition on the boundary $x^2=y^2$ and the Dirichlet boundary condition on the rest boundary of $\Omega$. The exact solution in the polar coordinates is given as follows
\[
\begin{pmatrix}u_r(r,\theta) \\ u_{\theta}(r,\theta)\end{pmatrix}=\frac{r^{\alpha}}{2\mu}\begin{pmatrix}-(\alpha+1)\cos((\alpha+1)\theta) + (C_2-\alpha-1)C_1\cos((\alpha-1)\theta) \\ (\alpha+1)\sin((\alpha+1)\theta) + (C_2+\alpha-1)C_1\sin((\alpha-1)\theta) \end{pmatrix}.
\]
The constants are $C_1:=-\cos((\alpha+1)\omega)/\cos((\alpha-1)\omega)$ and $C_2:=-2(\lambda+2\mu)/(\lambda+\mu)$, where $\alpha=0.544483736782$ is the positive solution of $\alpha \sin(2\omega)+\sin(2\omega\alpha)=0$ for $\omega=3\pi/4$.
The Lam\'{e} parameters
\[
\lambda=\frac{E\nu}{(1+\nu)(1-2\nu)},\quad \mu=\frac{E}{2(1+\nu)}
\]
with the elasticity modulus $E=10^5$ and the Poisson ratio $\nu=0.4999$.
The volume force $f(x,y)$ and the Neumann boundary data vanish, and the Dirichlet boundary condition is taken from the exact solution.
The histories of Algorithm~\ref{alg:amfem} for $k=3, 4, 5$ and $\vartheta=0.1$ are presented
in Figures~\ref{fig:errorex3priori}-\ref{fig:errorex3posteriori}, which indicate that the convergence rates of errors $\|\sigma-\sigma_h\|_A$ and $\eta(\sigma_h, \mathcal{T}_h)$ are both $O((\#\textrm{dofs})^{-(k+1)/2})$.
\begin{figure}[htbp]
  \centering
    \includegraphics[scale=0.8]{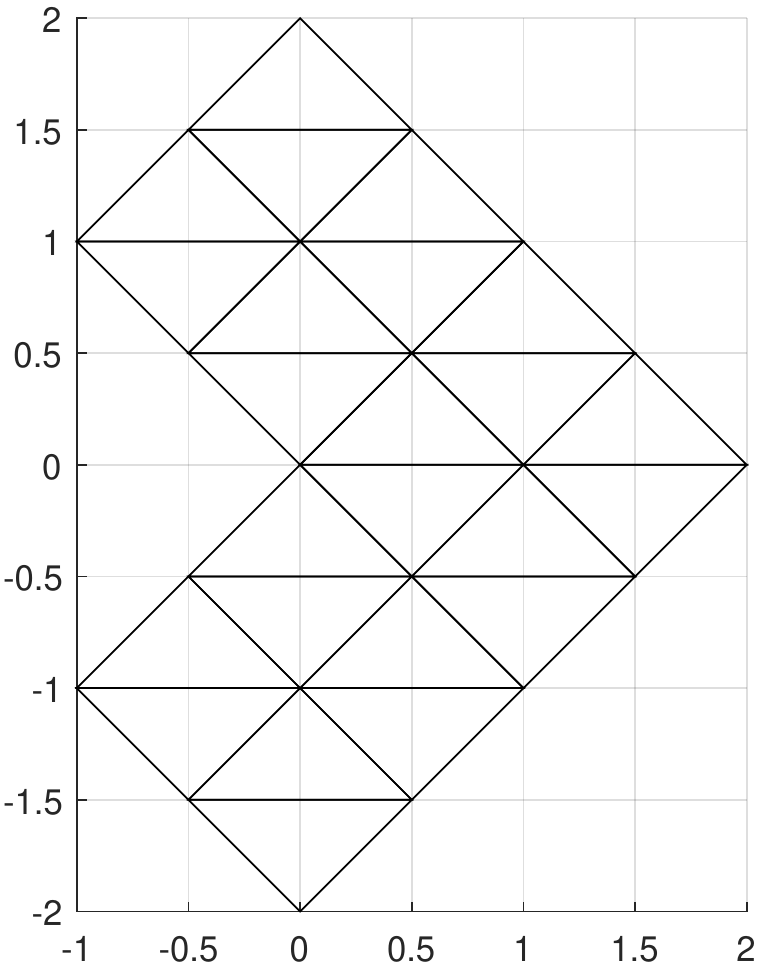}
  \caption{The rotated L-shaped domain with the initial mesh.} \label{fig:initialdomainex3}
\end{figure}
\begin{figure}[htbp]
  \centering
    \includegraphics[scale=0.8]{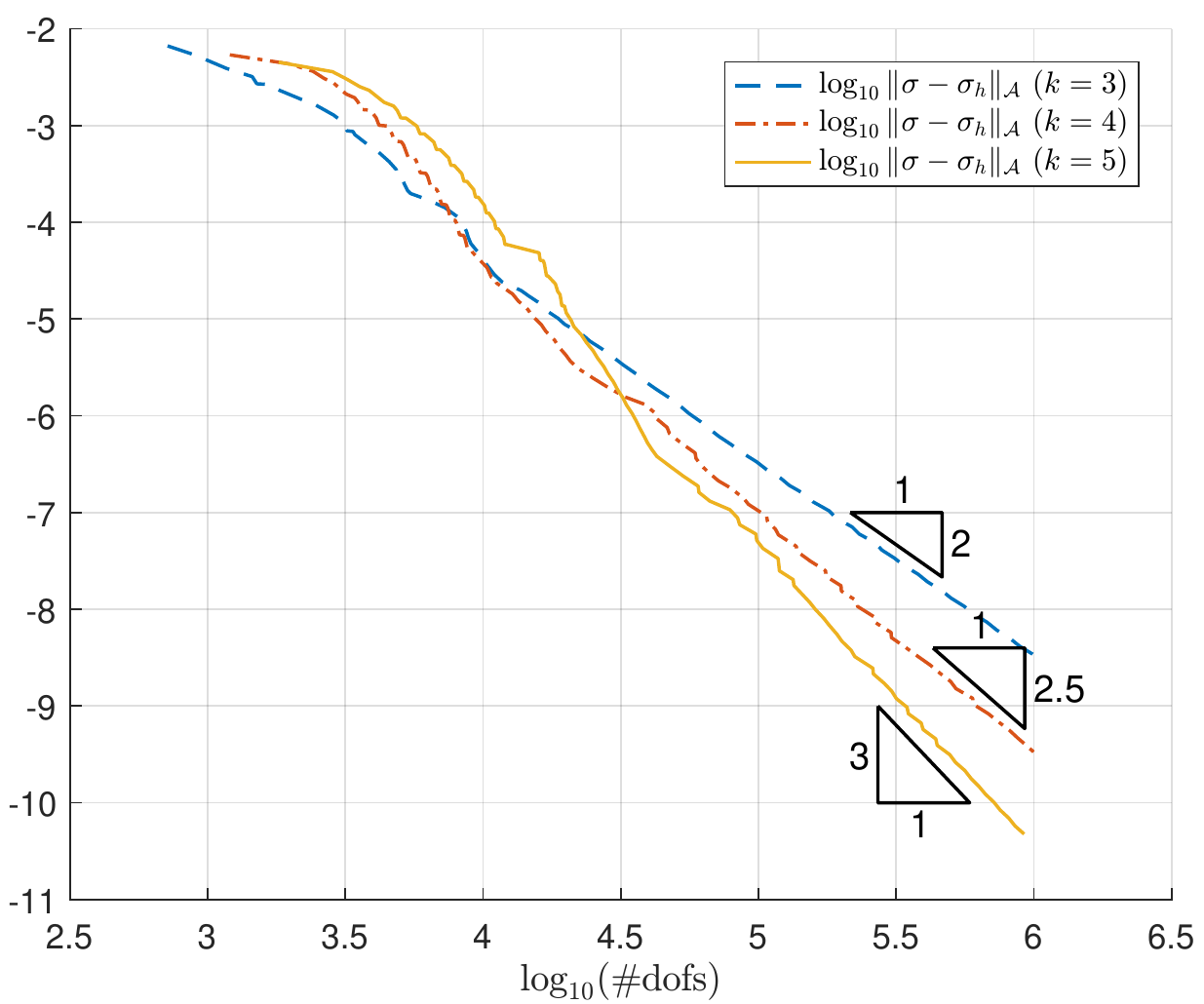}
  \caption{Errors $\|\sigma-\sigma_h\|_A$ vs $\#$dofs in
$\log_{10}$-$\log_{10}$ scale for Example 3 with $\vartheta=0.1$.} \label{fig:errorex3priori}
\end{figure}
\begin{figure}[htbp]
  \centering
    \includegraphics[scale=0.8]{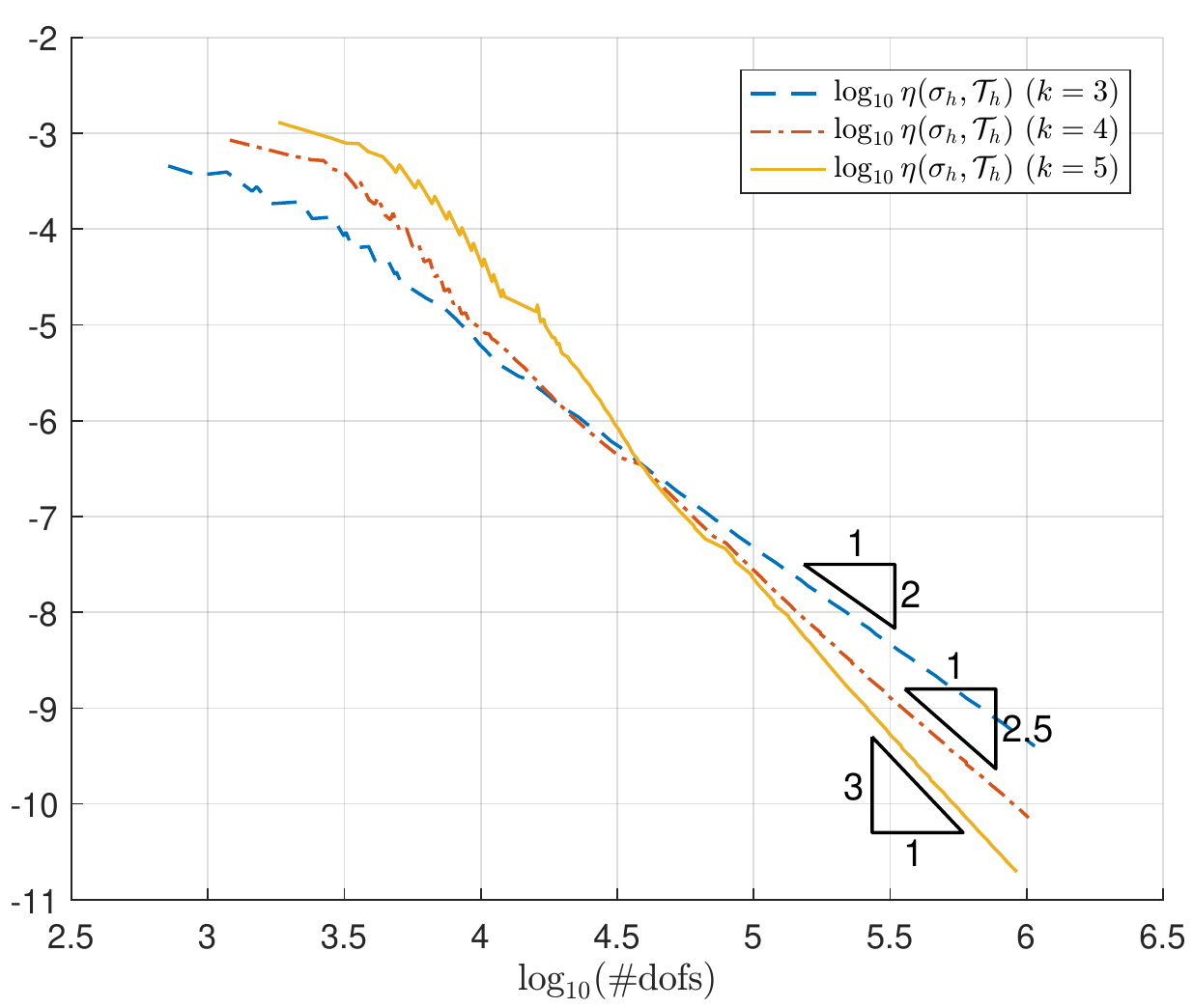}
  \caption{Errors $\eta(\sigma_h, \mathcal{T}_h)$ vs $\#$dofs in
$\log_{10}$-$\log_{10}$ scale for Example 3 with $\vartheta=0.1$.} \label{fig:errorex3posteriori}
\end{figure}


\end{document}